\documentclass[leqno,12pt]{amsart}
\usepackage{amssymb}
\usepackage{amsmath}
\usepackage{enumerate}
\usepackage{amsfonts}
\usepackage{hyperref}
\usepackage{mathrsfs}

\usepackage[headheight=18pt, top=25mm, bottom=25mm, left=25mm, right=25mm]{geometry}

\usepackage{tikz}

\definecolor{darkgreen}{RGB}{45, 119, 75}

\newtheorem{theorem}{Theorem}[section]
\newtheorem{corollary}[theorem]{Corollary}
\newtheorem{lemma}[theorem]{Lemma}
\newtheorem{proposition}[theorem]{Proposition}
\newtheorem{remark}[theorem]{Remark}

\numberwithin{equation}{section}

\hypersetup{
    colorlinks=true,
    linkcolor= blue,
    citecolor =cyan,
    urlcolor = teal,
}

\begin{document}

\title[Schr\"odinger semigroups]{On Dunkl Schr\"odinger semigroups \\
with Green bounded potentials}


\author[Jacek Dziubański]{Jacek Dziubański}
\author[Agnieszka Hejna]{Agnieszka Hejna}
\begin{abstract} On $\mathbb R^N$ equipped with a normalized root system $R$, a multiplicity function $k(\alpha) > 0$, and the associated measure 
 $$ 
dw(\mathbf x)=\prod_{\alpha\in R}|\langle \mathbf x,\alpha\rangle|^{k(\alpha)}\, d\mathbf x, 
$$ 
we consider a Dunkl Schr\"odinger operator $L=-\Delta_k+V$, where $\Delta_k$ is the Dunkl Laplace operator and $V\in L^1_{\rm loc} (dw)$ is a non-negative potential. 
Let $h_t(\mathbf x,\mathbf y)$ and $k^{\{V\}}_t(\mathbf x,\mathbf y)$ denote the Dunkl heat kernel and the integral kernel of the semigroup generated by $-L$ respectively.   We prove that $k^{\{V\}}_t(\mathbf x,\mathbf y)$ satisfies the following heat kernel lower bounds: there are constants $C, c>0$ such that 
$$ h_{ct}(\mathbf x,\mathbf y)\leq C k^{\{V\}}_t(\mathbf x,\mathbf y)$$ 
if and only if 
{$$ \sup_{\mathbf x\in\mathbb R^N} \int_0^\infty \int_{\mathbb R^N} V(\mathbf y)w(B(\mathbf x,\sqrt{t}))^{-1}e^{-\|\mathbf x-\mathbf y\|^2/t}\, dw(\mathbf y)\, dt<\infty, $$
where $B(\mathbf x,\sqrt{t})$ stands for the Euclidean ball centered at $\mathbf x \in \mathbb{R}^N$ and radius $\sqrt{t}$.}

\end{abstract}

\address{Jacek Dziubański, Uniwersytet Wroc\l awski,
Instytut Matematyczny,
Pl. Grunwaldzki 2,
50-384 Wroc\l aw,
Poland}
\email{jdziuban@math.uni.wroc.pl}

\address{Agnieszka Hejna, Uniwersytet Wroc\l awski,
Instytut Matematyczny,
Pl. Grunwaldzki 2,
50-384 Wroc\l aw,
Poland}
\email{hejna@math.uni.wroc.pl}

\thanks{
}

\maketitle

\section{Introduction and statement of the results}\label{Sec:Intro}
 
{Let $A=-\Delta +V$ be a Schr\"odinger operator on $\mathbb R^N$, $N\geq 3$. It is well-known (see~\cite{Sem}) that if $V\geq 0$, $V\in L^1_{\rm loc}(\mathbb R^N(dx))$, then the kernel $k_t(x,y)$ of the semigroup $\{e^{-tA}\}_{t \geq 0}$ satisfies the Gaussian (heat kernel) lower bounds 
$$ t^{-N/2} e^{-\| x-y\|^2/4ct}\leq Ck_t(x,y)$$
with certain constants $C,c>0$ if and only if the potential $V$ is Green bounded, that is,  
 $$ \sup_{x\in\mathbb R^N} \int_{\mathbb R^N} \frac{V(y)}{\| x-y\|^{N-2}}\, dy<\infty.$$
The aim of this paper is to prove similar results in the Dunkl setting. }

{
On the Euclidean space  $\mathbb R^N$ equipped with a normalized root system $R$ and a multiplicity function $k:R \longmapsto (0,\infty)$, let $\Delta_k$ denote the Dunkl Laplace operator (see Section \ref{sec:preliminaries}). Let $dw(\mathbf x)=w(\mathbf x)\, d\mathbf x$ be the associated measure, where 
 \begin{equation}\label{eq:measure}
w(\mathbf x)=\prod_{\alpha\in R}|\langle \mathbf x,\alpha\rangle|^{k(\alpha)}
\end{equation}
is its density with respect to the Lebesgue measure $d\mathbf x$. For a Lebesgue  measurable set $F\subseteq \mathbb R^N$, we denote
\begin{align}\label{eq:measure_w}
w(F):=\int_F dw(\mathbf x).    
\end{align}

It is well-known that $\Delta_k$ generates a semigroup  $\{H_t\}_{t \geq 0}=\{e^{t\Delta_k}\}_{t \geq 0}$ of linear operators on $L^2(dw)$ which has the form 
\begin{align*}
   {H_t}f(\mathbf x)=\int_{\mathbb{R}^N} h_t(\mathbf x,\mathbf y)f(\mathbf y)\, dw(\mathbf y),
\end{align*}
where $0<h_t(\mathbf x,\mathbf y)$ is a smooth function called the Dunkl heat kernel (see Section~\ref{sec:heat} for more details). 

Let $V\in L^1_{\rm loc}{(dw)}$ be a non-negative potential. Consider the Dunkl Schr\"odinger operator 
$$ L=-\Delta_k+V.$$
Then  $-L$ generates a semigroup $\{e^{-tL}\}_{t \geq 0}$ of self-adjoint linear contractions on $L^2(dw)$. The semigroup $\{e^{-tL}\}_{t \geq 0}$ has the form 
\begin{align*}
    e^{-tL}f(\mathbf x)=\int_{\mathbb R^N} k^{\{V\}}_t(\mathbf x,\mathbf y) f(\mathbf y)\, dw(\mathbf y),
\end{align*}
where the integral kernel $k^{\{V\}}_t(\mathbf x,\mathbf y)$ satisfies upper heat kernel bounds 
\begin{equation}\label{eq:k_very_basic}
    0 \leq k^{\{V\}}_t(\mathbf x,\mathbf y)\leq h_t(\mathbf x,\mathbf y).    
\end{equation}

The main goal of this paper is to characterize non-negative potentials $V\in L^1_{\rm loc}(dw)$ for which $k^{\{V\}}_t(\mathbf x,\mathbf y)$ satisfies the following heat kernel lower bound 
\begin{align*}
     h_{Ct}(\mathbf x,\mathbf y)\leq Ck^{\{V\}}_t(\mathbf x,\mathbf y).
\end{align*}
}

In order to state the result we need to introduce some notation. For $\alpha\in R$, let 
\begin{equation}\label{reflection}\sigma_\alpha (\mathbf x):=\mathbf x-2\frac{\langle \mathbf x,\alpha\rangle}{\|\alpha\|^2} \alpha
\end{equation}
be the reflection with respect to the subspace  perpendicular to $\alpha$. Let $G$ denote the reflection group generated by the reflections $\sigma_\alpha$, $\alpha\in R$. We define the distance of the orbit of $\mathbf x$ to the orbit of $\mathbf y$ by 
\begin{equation}\label{eq:d(x,y)}
    d(\mathbf x,\mathbf y)=\min \{ \| \mathbf x-\sigma (\mathbf y)\|: \sigma \in G\}.
\end{equation}
Obviously,
\begin{align*}
    d(\mathbf x,\mathbf y)=d(\mathbf x,\sigma (\mathbf y)) \quad \text{for all } \mathbf x,\mathbf y\in \mathbb R^N \ \text{\rm and } \sigma \in G.
\end{align*}
 Let
\begin{align*}
    B(\mathbf x,r)=\{ \mathbf x'\in \mathbb R^N: \| \mathbf x-\mathbf x'\|\leq r\}    
\end{align*}
stands for the (closed) Euclidean ball centered at $\mathbf x \in \mathbb{R}^N$ and radius $r>0$. 

Let $\mathbf{N}=N+\sum_{\alpha \in R}k(\alpha)$ be the homogeneous dimension of the system $(R,k)$. Throughout this paper we shall assume that $\mathbf N>2$ 

Our goal is to prove the following theorem. 
\begin{theorem}\label{teo:main}
Assume that $\mathbf N>2$ and $V\colon \mathbb{R}^N \longmapsto [0,\infty)$\textup{,} $V\in L^1_{\rm loc}(dw)$. Then the following are equivalent. 
\begin{enumerate}[(a)]
   \item\label{item:h_ct_leq_k_t}
   {The kernel $k^{\{V\}}_t(\mathbf{x},\mathbf{y})$ satisfies the following Dunkl heat kernel lower bound\textup{:} there are constants $C,c>0$ such that for all $\mathbf{x},\mathbf{y} \in \mathbb{R}^N$ and $t>0$ we have 
   $$ h_{ct}(\mathbf{x},\mathbf{y})\leq C k_t^{\{V\}}(\mathbf{x},\mathbf{y}). $$}
  
    \item \label{item:int_k_geq_delta} {There is a constant $\delta>0$ such that for all $\mathbf{x} \in \mathbb{R}^N$ and $t>0$ we have
    \begin{align*}
    \int_{\mathbb{R}^N} k^{\{V\}}_t(\mathbf{x},\mathbf{y})\, dw(\mathbf{y})\geq \delta.    
    \end{align*}
    }
    \item\label{item:potential_Green_bounded}{The potential $V$ is Green bounded, that is, 
   \begin{equation*}
         \sup_{\mathbf x\in\mathbb R^N} \int_0^\infty \int_{\mathbb R^N} V(\mathbf{y})h_s(\mathbf{x},\mathbf{y})\, dw(\mathbf{y})\, ds<\infty. \end{equation*}}

\end{enumerate}
\end{theorem}

{\begin{remark}\label{remark:GG} \normalfont
Condition~\eqref{item:potential_Green_bounded} of Theorem~\ref{teo:main} is equivalent to any of the following ones\textup{:} 
\begin{enumerate}
    \item[(c')]{\begin{equation*} \sup_{\mathbf{x} \in \mathbb{R}^N}\int_0^\infty \int_{\mathbb R^N}V(\mathbf{y})w(B(\mathbf{x},\sqrt{s}))^{-1} e^{-d(\mathbf{x},\mathbf{y})^2/s}\, dw(\mathbf{y})\, ds <\infty, 
   \end{equation*}}
   \item[(c'')]{\begin{equation*} \sup_{\mathbf{x} \in \mathbb{R}^N}\int_0^\infty \int_{\mathbb R^N}V(\mathbf{y})w(B(\mathbf{x},\sqrt{s}))^{-1} e^{-\|\mathbf x-\mathbf y\|^2/s}\, dw(\mathbf{y})\, ds <\infty. 
   \end{equation*} }
\end{enumerate}

The equivalences are proved in Proposition~\ref{prop:Greens}.
\end{remark}
}
The proof of Theorem~\ref{teo:main} depends very much  on the upper and lower bounds for $h_t(\mathbf x,\mathbf y)$ derived in~\cite{DH-heat}. We present them in Subsection \ref{sub:heat_bound}.

\section{Preliminaries}\label{sec:preliminaries}

\subsection{Basic definitions of the Dunkl theory}

In this section we present basic facts concerning the theory of the Dunkl operators.   For more details we refer the reader to~\cite{Dunkl},~\cite{Roesle99},~\cite{Roesler3}, and~\cite{Roesler-Voit}. 

We consider the Euclidean space $\mathbb R^N$ with the scalar product $\langle \mathbf{x},\mathbf y\rangle=\sum_{j=1}^N x_jy_j
$, where $\mathbf x=(x_1,\ldots,x_N)$, $\mathbf y=(y_1,\ldots,y_N)$, and the norm $\| \mathbf x\|^2=\langle \mathbf x,\mathbf x\rangle$.

A {\it normalized root system}  in $\mathbb R^N$ is a finite set  $R\subset \mathbb R^N\setminus\{0\}$ such that $R \cap \alpha \mathbb{R} = \{\pm \alpha\}$,  $\sigma_\alpha (R)=R$, and $\|\alpha\|=\sqrt{2}$ for all $\alpha\in R$, where $\sigma_\alpha$ is defined by \eqref{reflection}. 

The finite group $G$ generated by the reflections $\sigma_{\alpha}$, $\alpha \in R$, is called the {\it reflection group} of the root system. Clearly, $|G|\geq |R|$.

A~{\textit{multiplicity function}} is a $G$-invariant function $k:R\to\mathbb C$ which will be fixed and $> 0$  throughout this paper.  

Recall that $\mathbf{N}=N+\sum_{\alpha \in R}k(\alpha)$. Then, 
\begin{align*} w(B(t\mathbf x, tr))=t^{\mathbf{N}}w(B(\mathbf x,r)) \ \ \text{\rm for all } \mathbf x\in\mathbb R^N, \ t,r>0,
\end{align*}
where $w$ is the associated measure defined in \eqref{eq:measure_w}. 
Observe that there is a constant $C>0$ such that for all $\mathbf{x} \in \mathbb{R}^N$ and $r>0$ we have
\begin{equation}\label{eq:balls_asymp}
C^{-1}w(B(\mathbf x,r))\leq  r^{N}\prod_{\alpha \in R} (|\langle \mathbf x,\alpha\rangle |+r)^{k(\alpha)}\leq C w(B(\mathbf x,r)),
\end{equation}
so $dw(\mathbf x)$ is doubling, that is, there is a constant $C>0$ such that
\begin{equation}\label{eq:doubling} w(B(\mathbf x,2r))\leq C w(B(\mathbf x,r)) \ \ \text{ for all } \mathbf x\in\mathbb R^N, \ r>0.
\end{equation}
Let us also remark the the sets of measure zero with respect to the measure $dw(\mathbf x)$ and the Lebesgue measure $d\mathbf x$ coincide.  

For $\xi \in \mathbb{R}^N$, the {\it Dunkl operators} $T_\xi$  are the following $k$-deformations of the directional derivatives $\partial_\xi$ by   difference operators:
\begin{align*}
     T_\xi f(\mathbf x)= \partial_\xi f(\mathbf x) + \sum_{\alpha\in R} \frac{k(\alpha)}{2}\langle\alpha ,\xi\rangle\frac{f(\mathbf x)-f(\sigma_\alpha(\mathbf{x}))}{\langle \alpha,\mathbf x\rangle}.
\end{align*}
The Dunkl operators $T_{\xi}$, which were introduced in~\cite{Dunkl}, commute and are skew-symmetric with respect to the $G$-invariant measure $dw$. Let us denote $T_j=T_{e_j}$, where $\{e_j\}_{1 \leq j \leq N}$ is a canonical orthonormal  basis of $\mathbb{R}^N$.

\subsection{Dunkl kernel and Dunkl transform}

For fixed $\mathbf y\in\mathbb R^N$ the {\it Dunkl kernel} $E(\mathbf x,\mathbf y)$ is a unique analytic solution to the system
\begin{align*}
    T_\xi f=\langle \xi,\mathbf y\rangle f, \ \ f(0)=1.
\end{align*}
The function $E(\mathbf x ,\mathbf y)$, which generalizes the exponential  function $e^{\langle \mathbf x,\mathbf y\rangle}$, has a  unique extension to a holomorphic function on $\mathbb C^N\times \mathbb C^N$.

The \textit{Dunkl transform} is defined by
  \begin{equation}\label{eq:transform}\mathcal F f(\xi)={\boldsymbol c}_k^{-1}\int_{\mathbb R^N} E(-i\xi, \mathbf x)f(\mathbf x)\, dw(\mathbf x),
  \end{equation}
  where
  \begin{align*}
      {\boldsymbol c}_k=\int_{\mathbb{R}^N}e^{-\frac{\|\mathbf{x}\|^2}{2}}\,dw(\mathbf{x})>0,
  \end{align*}
for $f\in L^1(dw)$ and $\xi \in \mathbb{R}^N$. It was introduced in~\cite{D5} for $k \geq 0$ and further studied in~\cite{dJ1}. It was proved in~\cite[Corollary 2.7]{D5} (see also~\cite[Theorem 4.26]{dJ1}) that it is an isometry on $L^2(dw)$, i.e.,
   \begin{equation}\label{eq:Plancherel}
       \|f\|_{L^2(dw)}=\|\mathcal{F}f\|_{L^2(dw)} \text{ for all }f \in L^2(dw).
   \end{equation}

\subsection{Dunkl Laplacian and Dunkl heat semigroup}\label{sec:heat} The {\it Dunkl Laplacian} associated with $R$ and $k$  is the differential-difference operator $\Delta_k=\sum_{j=1}^N T_{j}^2$, which  acts on $C^2(\mathbb{R}^N)$-functions by\begin{align*}
    \Delta_k f(\mathbf x)=\Delta_{\rm eucl} f(\mathbf x)+\sum_{\alpha\in R} k(\alpha) \delta_\alpha f(\mathbf x),
    \quad
    \delta_\alpha f(\mathbf x)=\frac{\partial_\alpha f(\mathbf x)}{\langle \alpha , \mathbf x\rangle} - \frac{\|\alpha\|^2}{2} \frac{f(\mathbf x)-f(\sigma_\alpha (\mathbf x))}{\langle \alpha, \mathbf x\rangle^2}.
\end{align*}
The operator $\Delta_k$ is essentially self-adjoint on $L^2(dw)$ (see for instance \cite[Theorem 3.1]{Amri}) and generates a semigroup $\{H_t\}_{t>0}$  of linear self-adjoint contractions on $L^2(dw)$. The semigroup has the form
  \begin{equation}\label{eq:H}
  H_t f(\mathbf x)=\int_{\mathbb R^N} h_t(\mathbf x,\mathbf y)f(\mathbf y)\, dw(\mathbf y),
  \end{equation}
  where the heat kernel
  \begin{equation}\label{eq:heat_formula}
      h_t(\mathbf x,\mathbf y)={\boldsymbol c}_k^{-1} (2t)^{-\mathbf{N}/2}E\Big(\frac{\mathbf x}{\sqrt{2t}}, \frac{\mathbf y}{\sqrt{2t}}\Big)e^{-(\| \mathbf x\|^2+\|\mathbf y\|^2)\slash (4t)}
  \end{equation}
  is a $C^\infty$-function of the all variables $\mathbf x,\mathbf y \in \mathbb{R}^N$, $t>0$, and satisfies \begin{equation}\label{eq:positive} 0<h_t(\mathbf x,\mathbf y)=h_t(\mathbf y,\mathbf x),
  \end{equation}
  \begin{equation}
      \label{eq:probabilistic}
      \int_{\mathbb{R}^N} h_t(\mathbf x,\mathbf y)\, dw(\mathbf y)=1. 
  \end{equation}

The following specific formula for the Dunkl heat kernel was obtained by R\"osler \cite{Roesler2003}:
\begin{equation}\label{heat:Rosler}
h_t(\mathbf x,\mathbf y)={\boldsymbol c}_k^{-1}2^{-\mathbf{N}/2} t^{-\mathbf{N}/2}\int_{\mathbb R^N}  \exp(-A(\mathbf x, \mathbf y, \eta)^2/4t)\,d\mu_{\mathbf{x}}(\eta)\text{ for all }\mathbf{x},\mathbf{y}\in\mathbb{R}^N, 
 t>0.
\end{equation}
Here
\begin{equation}\label{eq:A}
A(\mathbf{x},\mathbf{y},\eta)=\sqrt{{\|}\mathbf{x}{\|}^2+{\|}\mathbf{y}{\|}^2-2\langle \mathbf{y},\eta\rangle}=\sqrt{{\|}\mathbf{x}{\|}^2-{\|}\eta{\|}^2+{\|}\mathbf{y}-\eta{\|}^2}
\end{equation}
and $\mu_{\mathbf{x}}$ is a probability measure,
which is supported in the convex hull  $\operatorname{conv}\mathcal{O}(\mathbf{x})$ of the orbit  $\mathcal O(\mathbf x) =\{\sigma(\mathbf x): \sigma \in G\}$.

\subsection{Upper and lower heat kernel bounds}\label{sub:heat_bound}

The closures of connected components of
\begin{equation*}
    \{\mathbf{x} \in \mathbb{R}^{N}\;:\; \langle \mathbf{x},\alpha\rangle \neq 0 \text{ for all }\alpha \in R\}
\end{equation*}
are called (closed) \textit{Weyl chambers}.

\begin{lemma}\label{lem:realization_of_d} 
Fix $\mathbf x,\mathbf y\in\mathbb R^N$ and $\sigma \in G$. Then  $d(\mathbf x,\mathbf y)=\| \mathbf{x}-\sigma(\mathbf y)\|$ if and only if  $\sigma(\mathbf y)$ and $\mathbf x$ belong to the same Weyl chamber. 
\end{lemma}
\begin{proof}
See \cite[Chapter VII, proof of Theorem 2.12]{Helgason}. 
\end{proof}

For a finite  sequence $\boldsymbol \alpha  =(\alpha_1,\alpha_2,\ldots,\alpha_m)$ of elements of $R$, $\mathbf x,\mathbf y\in \mathbb R^N$ and $t>0$,  let 

\begin{equation}
    \ell(\mathbf{\boldsymbol \alpha}):=m 
\end{equation}
be the length of $\boldsymbol \alpha$, 
\begin{equation}
    \sigma_{\boldsymbol \alpha}:=\sigma_{\alpha_m}\circ \sigma_{\alpha_{m-1}}\circ \ldots\circ\sigma_{\alpha_1}, 
\end{equation}
and 
 \begin{equation}\label{eq:rho}\begin{split}
    &\rho_{\boldsymbol \alpha}(\mathbf{x},\mathbf{y},t)\\&:=\left(1+\frac{\|\mathbf{x}-\mathbf{y}\|}{\sqrt{t}}\right)^{-2}\left(1+\frac{\|\mathbf{x}-\sigma_{\alpha_1}(\mathbf{y})\|}{\sqrt{t}}\right)^{-2}\left(1+\frac{\|\mathbf{x}-\sigma_{\alpha_2} \circ \sigma_{\alpha_1}(\mathbf{y})\|}{\sqrt{t}}\right)^{-2}\times \ldots  \\&\times \left(1+\frac{\|\mathbf{x}-\sigma_{\alpha_{m-1}} \circ \ldots \circ \sigma_{\alpha_1}(\mathbf{y})\|}{\sqrt{t}}\right)^{-2}.
    \end{split}
\end{equation}

For $\mathbf x,\mathbf y\in\mathbb R^N$, let 
    $ n (\mathbf x,\mathbf y)=0$ if $  d(\mathbf x,\mathbf y)=\| \mathbf x-\mathbf y\|$ and  
    
    \begin{equation}\label{eq:n}
        n(\mathbf x,\mathbf y) = 
    \min\{m\in\mathbb Z: d(\mathbf x,\mathbf y)=\| \mathbf x-\sigma_{\alpha_{m}}\circ \ldots \circ \sigma_{\alpha_2}\circ\sigma_{\alpha_1}(\mathbf y)\|,\quad \alpha_j\in R\}     
    \end{equation}
otherwise. In other words, $n(\mathbf x,\mathbf y)$ is the smallest number of reflections $\sigma_\alpha$ which are needed to move $\mathbf y$ to the (closed) Weyl chamber of $\mathbf x$. 
We also allow  $\boldsymbol{\alpha}$ to be the empty sequence, denoted by $\boldsymbol{\alpha} =\emptyset$. Then for $\boldsymbol{\alpha}=\emptyset$, we set:   $\sigma_{\boldsymbol{\alpha}}=I$ (the identity operator), $\ell(\boldsymbol{\alpha})=0$, and $\rho_{\boldsymbol{\alpha}}(\mathbf{x},\mathbf{y},t)=1$ for all $\mathbf{x},\mathbf{y} \in \mathbb{R}^N$ and $t>0$. 
   
We say that  a finite  sequence $\boldsymbol \alpha  =(\alpha_1,\alpha_2,\ldots,\alpha_m)$ of  roots is {\it admissible for the pair}  $(\mathbf x,\mathbf y)\in\mathbb R^N\times\mathbb R^N$  if $n(\mathbf{x},\sigma_{\boldsymbol{\alpha}}(\mathbf{y}))=0$. In other words, the composition $\sigma_{\alpha_m}\circ\sigma_{\alpha_{m-1}}\circ \ldots\circ \sigma_{\alpha_1}$ of the reflections $\sigma_{\alpha_j}$ maps $\mathbf y$ to the Weyl chamber of $\mathbf x$. 
The set of the all admissible sequences $\boldsymbol \alpha$ for the pair  $(\mathbf x,\mathbf y)$  will be denoted by $\mathcal A(\mathbf x,\mathbf y)$.  
Note that if $n(\mathbf x,\mathbf y)=0$, then $\boldsymbol{\alpha}=\emptyset \in \mathcal A(\mathbf{x},\mathbf{y})$.

Let us define
\begin{equation}\label{eq:Lambda_def}
    \Lambda (\mathbf x,\mathbf y,t):=\sum_{\boldsymbol \alpha \in \mathcal{A}(\mathbf{x},\mathbf{y}), \;\ell (\boldsymbol \alpha) \leq 2|G|} \rho_{\boldsymbol \alpha}(\mathbf x,\mathbf y,t).
\end{equation} 
Note that for any $c>1$ and for all $\mathbf{x},\mathbf{y} \in \mathbb{R}^N$ and $t>0$ we have
\begin{equation}\label{eq:Lambda_2t}
c^{-2|G|} \Lambda(\mathbf x,\mathbf y,ct)\leq  \Lambda(\mathbf x,\mathbf y,t)\leq \Lambda(\mathbf x,\mathbf y,ct).
\end{equation}

The following upper and lower bounds for $h_t(\mathbf x,\mathbf y)$ were proved in \cite{DH-heat}.

\begin{theorem}\label{teo:1}
Assume that $0<c_{u}<1/4$ and $c_l>1/4$. There are constants $C_{u},C_{l}>0$ such that for all $\mathbf{x},\mathbf{y} \in \mathbb{R}^N$ and $t>0$ we have 
\begin{equation}\label{eq:main_lower}
 C_{l}w(B(\mathbf{x},\sqrt{t}))^{-1}e^{-c_{l}\frac{d(\mathbf{x},\mathbf{y})^2}{t}} \Lambda(\mathbf x,\mathbf y,t) \leq    h_t(\mathbf{x},\mathbf{y}), 
\end{equation}
\begin{equation}\label{eq:main_claim}
    h_t(\mathbf{x},\mathbf{y}) \leq C_{u}w(B(\mathbf{x},\sqrt{t}))^{-1}e^{-c_{u}\frac{d(\mathbf{x},\mathbf{y})^2}{t}} \Lambda(\mathbf x,\mathbf y,t).
\end{equation}
\end{theorem}
{ 
\begin{remark}\normalfont
In Theorem~\ref{teo:1}, we can replace $\Lambda(\mathbf{x},\mathbf{y},t)$ by the function 
  \begin{equation}\label{eq:Lambda_def1}
    \tilde \Lambda (\mathbf x,\mathbf y,t):=\sum_{\boldsymbol \alpha \in \mathcal{A}(\mathbf{x},\mathbf{y}), \;\ell (\boldsymbol \alpha) \leq |G|} \rho_{\boldsymbol \alpha}(\mathbf x,\mathbf y,t).
\end{equation}
Indeed,  $\tilde \Lambda (\mathbf x,\mathbf y,t) \leq  \Lambda (\mathbf x,\mathbf y,t)$ for all $\mathbf{x},\mathbf{y} \in \mathbb{R}^N$ and $t>0$. We turn to prove
\begin{equation}\label{eq:lambda_lambda}
    \Lambda (\mathbf x,\mathbf y,t) \leq |R|^{2|G|}\tilde \Lambda (\mathbf x,\mathbf y,t).
\end{equation}
To this end, fix $\mathbf{x},\mathbf{y} \in \mathbb{R}^N$, $t>0$, and take $\boldsymbol \beta \in \mathcal {A}(\mathbf x,\mathbf y)$ of the minimal length $\ell(\boldsymbol\beta)$ which satisfies $\ell (\boldsymbol \beta)\leq 2|G|$, and  
\begin{equation}\label{eq:max_Lambda} \rho_{\boldsymbol \beta} (\mathbf x,\mathbf y,t)=\max_{\boldsymbol \alpha \in \mathcal{A}(\mathbf{x},\mathbf{y}), \;\ell (\boldsymbol \alpha) \leq 2|G|}\rho_{\boldsymbol \alpha}(\mathbf x,\mathbf y,t).  
\end{equation}
Obviously, $\Lambda(\mathbf x,\mathbf y,t)\leq |R|^{2|G|}\rho_{\boldsymbol \beta}(\mathbf x,\mathbf y,t)$. If $|\boldsymbol{\beta}| \leq |G|$, then~\eqref{eq:lambda_lambda} is proved. If $m=|\boldsymbol{\beta}|>|G|$, then let us consider the sequence 
\begin{equation}\label{eq:sequence}
    I, \sigma_{\beta_1}, \sigma_{\beta_2}\circ \sigma_{\beta_1}, \dots , \sigma_{\beta_{m-1}}\circ \dots\circ \sigma_{\beta_1}.
\end{equation}
Since there are $|\boldsymbol{\beta}|>|G|$ elements in the sequence~\eqref{eq:sequence}, at least two of them coincide.  

Assume first that for some $j,s \in \{1,2,\ldots,m-1\}$, $j<s$, we have
\begin{align*}
     \sigma_{\beta_j}\circ \sigma_{\beta_{j-1}}\circ \dots\circ  \sigma_{\beta_1}= \sigma_{\beta_s}\circ\dots \circ \sigma_{\beta_j} \circ  \sigma_{\beta_{j-1}}\circ \ldots \circ\sigma_{\beta_1}\ne I. 
\end{align*}
Set $\widetilde{\boldsymbol \beta}=(\beta_1,\beta_2,\ldots,\beta_j,\beta_{s+1},\ldots,\beta_m).$ Then $\widetilde{\boldsymbol \beta}\in\mathcal A(\mathbf x,\mathbf y)$, $\ell (\widetilde{\boldsymbol\beta})<\ell (\boldsymbol \beta)$, and 
$ \rho_{\boldsymbol \beta}(\mathbf x,\mathbf y,t)\leq \rho_{\widetilde{\boldsymbol \beta}}(\mathbf x,\mathbf y,t)$. 

If there is $s\in \{1,2,\ldots,m-1\}$ such that  $\sigma_{\beta_s}\circ\dots \circ \sigma_{\beta_j} \circ  \sigma_{\beta_{j-1}}\circ \ldots \circ \sigma_{\beta_1}=I$, then we set $\widetilde{\boldsymbol\beta}=(\beta_{s+1},\ldots,\beta_m)$ and argue as above. Thus \eqref{eq:Lambda_def1} is established. 
\end{remark}
}

In order to obtain our regularity results, we will need the following theorem proved in~\cite{DH-heat}.

\begin{theorem}[\cite{DH-heat}, Theorem 6.1]\label{teo:holder_est}
There are constants ${C}_4,{c}_4>0$ such that for all $\mathbf{x},\mathbf{y},\mathbf{y}' \in \mathbb{R}^N$ and $t>0$ satisfying $\|\mathbf{y}-\mathbf{y}'\|<\frac{\sqrt{t}}{2}$ we have
\begin{equation}\label{eq:h_t_Holder_d}
    |h_t(\mathbf{x},\mathbf{y})- h_{t}(\mathbf{x},\mathbf{y}')| \leq {C}_4\frac{\|\mathbf{y}-\mathbf{y}'\|}{\sqrt{t}}h_{{c}_4t}(\mathbf{x},\mathbf{y}).
\end{equation}
\end{theorem}

We will also need some auxiliary estimates of the generalized heat kernel $h_t(\mathbf{x},\mathbf{y})$.

\begin{lemma}\label{lem:small_deformation}
Assume that $c_0>1$. There is a constant $C_0>0$ such that for all $\mathbf{x},\mathbf{y},\mathbf{y}' \in \mathbb{R}^N$ and $t>0$ satisfying $\|\mathbf{y}-\mathbf{y}'\|< \sqrt{t}$ we have
\begin{equation}\label{eq:small_deformation}
    h_t(\mathbf{x},\mathbf{y}) \leq C_0h_{c_0t}(\mathbf{x},\mathbf{y}').
\end{equation}
\end{lemma}
\begin{proof}
This is Lemma 6.2 of \cite{DH-heat}. For the convenience of the reader, we present an alternative proof here. Let $\mathbf{y},\mathbf{y}' \in \mathbb{R}^N$ be such that $  \| \mathbf y-\mathbf y'\|\leq \sqrt{t}$. Recall that $ \| \mathbf x\| -\|\mathbf \eta\| \geq 0$ for all $\eta\in \text{\rm conv}\, \mathcal O(\mathbf x )$. Put $\varepsilon=c_0-1$.  We turn to estimate $A(\mathbf{x},\mathbf{y}',\eta)$ defined in~\eqref{eq:A}: 
 \begin{equation}\label{eq:A_comp}
     \begin{split}
         A(\mathbf x,\mathbf y',\eta)^2&=\| \mathbf x\|^2-\| \eta\|^2+\|\mathbf y'-\eta\|^2\\
         & \leq \| \mathbf x\|^2-\| \eta\|^2+(\|\mathbf y'-\mathbf y\|+\|\mathbf y-\eta\|)^2\\
         &\leq  \| \mathbf x\|^2-\| \eta\|^2+\|\mathbf y'-\mathbf y\|^2+\|\mathbf y-\eta\|^2+\varepsilon^{-1} \| \mathbf y'-\mathbf y\|^2+\varepsilon \|\mathbf y-\eta\|^2 \\
         & \leq (1+\varepsilon) (\| \mathbf x\|^2-\| \eta\|^2+\|\mathbf y-\eta\|^2)+ (1+\varepsilon^{-1})t \\
         &=(1+\varepsilon) A(\mathbf x,\mathbf y,\eta)^2+(1+\varepsilon^{-1})t,
     \end{split}
 \end{equation}
 where in the second inequality of~\eqref{eq:A_comp} we have used the inequality $2ab \leq \varepsilon a^2+\varepsilon^{-1}b^2$. 
Using \eqref{eq:A_comp} and the R\"osler formula~\eqref{heat:Rosler}, we get 
\begin{equation}
    \begin{split}
        h_t(\mathbf x,\mathbf y)&= {\boldsymbol c}_k^{-1}2^{-\mathbf{N}/2} t^{-\mathbf{N}/2} 
       \int_{\mathbb{R}^N} e^{-A(\mathbf x,\mathbf y,\eta)^2/4t}\, d\mu_{\mathbf x}(\eta)\\
        &\leq {\boldsymbol c}_k^{-1}2^{-\mathbf{N}/2} t^{-\mathbf{N}/2} \int_{\mathbb{R}^N} e^{-A(\mathbf x,\mathbf y',\eta)^2/4(1+\varepsilon)t } e^{(1+\varepsilon^{-1})/(4(1+\varepsilon))}\, d\mu_{\mathbf x}(\eta)\\
        &=C_{0} h_{c_0t}(\mathbf x,\mathbf y'). 
    \end{split}
\end{equation}
\end{proof}

As a consequence of Lemma~\ref{lem:small_deformation}, we obtain the next lemma, which will be used in the proof of the main theorem.

\begin{lemma}\label{lem:iter_lip}
There are constants $C,c>1$ such that for all $s,t>0$ and $\mathbf{x},\mathbf{x}',\mathbf{y} \in \mathbb{R}^N$ we have
\begin{equation}\label{eq:iter_lip}
\begin{split}
    \int_{\mathbb{R}^N}\Big| h_s(\mathbf x,\mathbf z)-h_s(\mathbf x',\mathbf z)\Big|  h_t(\mathbf{z},\mathbf{y})\,dw(\mathbf{z}) &\leq C \frac{\|\mathbf{x}-\mathbf{x}'\|}{\sqrt{s}}h_{cs+t}(\mathbf{x},\mathbf{y})+ C \frac{\|\mathbf{x}-\mathbf{x}'\|}{\sqrt{s}}h_{cs+t}(\mathbf{x}',\mathbf{y}).
\end{split}
\end{equation}
\end{lemma}

\begin{proof}
If $\|\mathbf{x}-\mathbf{x}'\| \geq \frac{\sqrt{s}}{2}$, then~\eqref{eq:iter_lip} follows by the semigroup property of $h_t(\mathbf x,\mathbf y)$. 
Assume that $\|\mathbf{x}-\mathbf{x}'\| < \frac{\sqrt{s}}{2}$.  Theorem \ref{teo:holder_est} asserts  that there are constants $C,c>1$ such that for all $s_1>0$ and $\mathbf{x}_1,\mathbf{x}'_1,\mathbf{z}_1 \in \mathbb{R}^N$ satisfying $\|\mathbf{x}_1-\mathbf{x}'_1\| <\frac{\sqrt{s_1}}{2}$ we have
\begin{equation}
    |h_{s_1}(\mathbf{x}_1,\mathbf{z}_1)-\ h_{s_1}(\mathbf{x}_1',\mathbf{z}_1)| \leq C\frac{\|\mathbf{x}_1-\mathbf{x}_1'\|}{\sqrt{s_1}}h_{cs_1}(\mathbf{x}_1,\mathbf{z}_1).
\end{equation}
Hence, by the semigroup property of the generalized heat semigroup, we obtain
\begin{align*}
     \int_{\mathbb{R}^N} \left| h_s(\mathbf{x},\mathbf{z})-h_s(\mathbf{x}',\mathbf{z})
     \right|
     h_t(\mathbf{z},\mathbf{y})\,dw(\mathbf{z}) &\leq C\frac{\|\mathbf{x}-\mathbf{x}'\|}{\sqrt{s}}\int_{\mathbb{R}^N}h_{cs}(\mathbf{x},\mathbf{z})h_{t}(\mathbf{z},\mathbf{y})\,dw(\mathbf{z})\\&=C\frac{\|\mathbf{x}-\mathbf{x}'\|}{\sqrt{s}}h_{cs+t}(\mathbf{x},\mathbf{y}).
\end{align*}
\end{proof}

\section{Dunkl Schr\"odinger operators with non-negative potentials - introductory results }\label{S-intro}

For a nonnegative potential $V \colon \mathbb{R}^N \longmapsto [0,\infty)$, {$V\in L^1_{\rm loc}(dw)$, let $V_n(x):=\min (V(\mathbf x),n)$, $n=1,2,\dots $. } We consider the quadratic forms  
\begin{equation}\label{eq:QQ}
    Q_\infty(f,g):=\int_{\mathbb{R}^N} \Big(\sum_{j=1}^N T_j f(\mathbf{x})\overline{T_jg(\mathbf{x})} + V(\mathbf{x})f(\mathbf{x})\overline{g(\mathbf{x})}\Big)\, dw(\mathbf{x}),
\end{equation}

\begin{equation}
    Q_n(f,g):=\int_{\mathbb{R}^N} \Big(\sum_{j=1}^N T_j f(\mathbf{x})\overline{T_jg(\mathbf{x})} + V_n(\mathbf{x})f(\mathbf{x})\overline{g(\mathbf{x})}\Big)\, dw(\mathbf{x}), 
\end{equation}
with the domains
\begin{align*}
    \mathcal D(Q_\infty)=\{ f\in L^2(dw): \| \mathbf{x}\|\mathcal F f(\mathbf{x})\in L^2(dw(\mathbf{x}))\  \text{ and }\  \sqrt{V(\mathbf{x})}f(\mathbf{x})\in L^2(dw(\mathbf{x}))\},
\end{align*}
\begin{align*}
     \mathcal D(Q_n)=\{ f\in L^2(dw): \| \mathbf{x}\|\mathcal F f(\mathbf{x})\in L^2(dw(\mathbf{x}))\ \},\quad n=1,2,\dots .
\end{align*}

Observe that $C_c^\infty(\mathbb R^N)$ is a dense subspace of $L^2(dw)$ such that   $C_c^\infty(\mathbb R^N)\subseteq \mathcal D(Q_\infty)\subseteq \mathcal D(Q_n)$.  
The forms $Q_\infty$ and $Q_n$  are non-negative and closed.  So they define  self-adjoint non-negative operators $L_\infty$, $L_n$ respectively: 
\begin{align*}
    \mathcal D(L_n)=\{ f\in \mathcal D(Q_n): |Q_n(f,g)|\leq C_f\| g\|_{L^2(dw)}\quad \text{for all } g\in \mathcal D(Q_n), \}, \quad n=\infty, 1,2,\dots
\end{align*}
and, for $f\in\mathcal D(L_n)$, the operator $L_n$ is defined by the equation 
\begin{align*}
    \int_{\mathbb{R}^N} (L_nf)(\mathbf{x})\bar g(\mathbf{x})\, dw(\mathbf{x})=Q_n(f,g) \quad \text{for all } g\in\mathcal D(Q_n),
\end{align*}
see e.g.~\cite[Theorem 4.12]{Davies}. 
{Moreover, $f\in \mathcal D(Q_\infty)$ if and only if $\lim_{n\to\infty} Q_n(f,f)<\infty$. Further,   $Q_\infty(f,f)=\lim_{n\to\infty} Q_n(f,f)$ and, by the definition of $V_n$, the convergence is monotone. Set $L:=L_\infty$.  
The operator $-L_n$ is the generator of a semigroup of linear contractions on $L^2(dw)$, denoted by $\{e^{-tL_n}\}_{t \geq 0}$ for  $n=1,2,\dots$ and $\{e^{-tL}\}_{t \geq 0}$ for $L=L_\infty$. Let $a>0$. 
Theorem 4.32 of \cite{Davies}
 asserts that 
\begin{equation}\label{eq:basic_convergence} \lim_{n\to\infty} \Big\{ \sup_{0\leq t\leq a}\| e^{-tL_n} f-e^{-tL}f\|_{L^2(dw)} \Big\} =0\quad \text{for all } f\in L^2(dw). 
\end{equation}}

In the forthcoming sections we provide  rigorous proofs of   existence, regularity and bounds for the kernels  of the semigroups $\{e^{-tL_n}\}_{t \geq 0}$, $n=\infty,1,2,\dots $. The  main tools are the following 
product formula and Duhamel formula for  semigroups generated by  perturbations of  generators by  bounded  operators on  Banach spaces which we state as the theorem. 

\begin{theorem}\label{teo:AB}
Let $A$ be a generator of a semigroup $\{e^{tA}\}_{t \geq 0}$ of linear operators on a Banach space $X$\textup{,} and let $B$ be a bounded operator on $X$. Then $A+B$ is a generator of a semigroup of linear operators on $X$,  denoted by $\{e^{t(A+B)}\}_{t \geq 0}$, and for every $x\in X$ one has 
\begin{equation}\label{eq:product-formula}
    e^{t(A+B)} x=\lim_{n\to\infty} \Big(e^{tA/n}e^{tB/n}\Big)^n x\textup{,}
\end{equation}
\begin{equation}
    e^{tA}x  = e^{t(A+B)}x -\int_0^t e^{(t-s)A}B e^{s(A+B)}x\, ds.
\end{equation}
Moreover\textup{,} if the semigroup $\{e^{tA}\}_{t \geq 0}$ is holomorphic\textup{,}  so is $\{e^{t(A+B)}\}_{t \geq 0}$. 
\end{theorem} 

 Let us remark that under a stronger assumption, namely $V\in L^2_{\rm loc}(dw)$\textup{,} it was proved in Amri and Hammi \cite{Amri}  that $L$ is essentially self-adjoint non-negative operator\textup{,} that is, $L$ is the closure of the operator  $$\mathcal L_\infty f=-\Delta_k f+Vf,$$ 
initially defined on $C_c^\infty(\mathbb R^N)$. We 
 will not use this assumption in our forthcoming considerations and keep the weaker assumption $V\in L^1_{\rm loc}( dw)$.

\section{Schr\"odinger semigroups with bounded potentials } 

In this section we utilize the product formula~\eqref{eq:product-formula} to get existence and regularity of the kernel $k^{\{V\}}_t(\mathbf x,\mathbf y)$  from properties of approximation kernels. 
In this section, we assume that $V\geq 0$ is a bounded potential. 

\begin{theorem}\label{teo:continuous_k}
Assume that $V:\mathbb{R}^N \longmapsto [0,\infty)$ is a bounded measurable function. Then the semigroup $\{e^{t(\Delta_k-V)}\}_{t \geq 0}$ of linear operators generated by $\Delta_k-V$ has the form
\begin{align}\label{eq:semigroup_form}
    e^{t(\Delta_k-V)}f(\mathbf x)=\int_{\mathbb{R}^N} k^{\{V\}}_t(\mathbf x,\mathbf y)f(\mathbf y)\, dw(\mathbf y),\quad f\in L^2(dw),
\end{align}
where $\mathbb{R}^N \times \mathbb{R}^N \times (0,\infty) \ni (\mathbf{x},\mathbf{y},t) \longmapsto k^{\{V\}}_t(\mathbf x,\mathbf y)$ is a continuous function such that there are constants $C,c>0$ such that
\begin{align*}
    0\leq  k^{\{V\}}_t(\mathbf x,\mathbf y)\leq h_t(\mathbf x,\mathbf y),
\end{align*}
\begin{equation}\label{eq:lipschitz_k_t} |k^{\{V\}}_t(\mathbf x,\mathbf y)-k^{\{V\}}_t(\mathbf x',\mathbf y')|\leq C (1+\sqrt{t\|V\|_\infty})\frac{\| \mathbf x-\mathbf x'\|+\|\mathbf y-\mathbf y'\|}{\sqrt{t}}h(ct,\mathbf x,\mathbf x',\mathbf y,\mathbf y'),
\end{equation}
for all $\mathbf{x},\mathbf{x}',\mathbf{y},\mathbf{y}' \in \mathbb{R}^N$ and $t>0$, 
where
\begin{align*}
    h(ct,\mathbf x_1,\mathbf x_2,\mathbf y_1,\mathbf y_2):=\sum_{i=1}^2\sum_{j=1}^2h_{ct}(\mathbf x_i,\mathbf y_j).    
\end{align*}

Moreover, for all $\mathbf x,\mathbf y\in\mathbb R^N$, the function $(0,\infty)\ni t \to k^{\{V\}}_t(\mathbf x,\mathbf y)$ is differentiable and for any $m \in \mathbb{N}$ there is a constant $C_m>0$ such that for all $\mathbf x, \mathbf y \in \mathbb{R}^N$ and $t>0$ we have
\begin{equation}\label{eq:time_der} \Big|\frac{d^m}{dt^m} k^{\{V\}}_t(\mathbf x,\mathbf y)\Big|\leq C_mt^{-m} w(B(\mathbf x,\sqrt{t}))^{-1/2} w(B(\mathbf y,\sqrt{t}))^{-1/2}.
\end{equation}
The constants $C,c,C_m$ are independent of $V$.
\end{theorem}

\begin{proof} We assume that $V\not\equiv 0$. 
It suffices to prove \eqref{eq:lipschitz_k_t} for $0<t\leq \| V\|_{\infty}^{-1} $ and then use the semigroup property. 

Let us consider the integral kernels $Q_{n,t}(\mathbf x,\mathbf y)$ of the operators $(H_{t/n}e^{-tV/n})^n$.  
We write 

\begin{align}\label{eq:q_n_t}
    Q_{n,t}(\mathbf x,\mathbf y):=q_{n,t}(\mathbf x,\mathbf y)e^{-tV(\mathbf y)/n},
\end{align}
where 
\begin{equation}\label{eq:q_n_t_1}
    \begin{split}
       & 0\leq  q_{n,t} (\mathbf x,\mathbf y)\\
       &:=\int_{\mathbb{R}^N}\ldots\int_{\mathbb{R}^N} h_{t/n}(\mathbf x,\mathbf z_1)e^{-tV(\mathbf z_1)/n}h_{t/n}(\mathbf z_1,\mathbf z_2)e^{-tV(\mathbf z_2)/n}\ldots 
       h_{t/n} (\mathbf z_{n-1}, \mathbf y)\, dw(\mathbf z_{n-1})\ldots dw(\mathbf z_1)\\
       &\leq h_t(\mathbf x,\mathbf y).
    \end{split} 
\end{equation}
We prove that the functions $\mathbb R^N \times \mathbb R^N\ni (\mathbf x,\mathbf y)\mapsto q_{n,t}(\mathbf x,\mathbf y)$, which are clearly continuous,  are Lipschitz functions of $(\mathbf x,\mathbf y)$. A uniform bound independent of $n \in \mathbb{N}$ will be given.

For $\mathbf{z}_1 \in \mathbb{R}^N$ we write  $\exp(-tV(\mathbf z_1)/n)=1-tW(\mathbf z_1)/n$ 
where $|W(\mathbf z_1)|\leq \| V\|_{\infty}$. 
Thus, thanks to the fact $\int_{\mathbb{R}^N}h_{t/n}(\mathbf{x},\mathbf{z}_1)h_{t/c}(\mathbf{z}_1,\mathbf{z}_2)\,dw(\mathbf{z}_1)=h_{2t/n}(\mathbf{x},\mathbf{z}_2)$, we get
\begin{equation*}
    \begin{split}
       &  q_{n,t} (\mathbf x,\mathbf y)\\
       &=\int_{\mathbb{R}^N}\ldots\int_{\mathbb{R}^N} h_{2t/n}(\mathbf x,\mathbf z_2)e^{-tV(\mathbf z_2)/n}h_{t/n}(\mathbf z_2,\mathbf z_3)e^{-tV(\mathbf z_3)/n}\ldots 
       h_{t/n} (\mathbf z_{n-1}, \mathbf y)\, dw(\mathbf z_{n-1})\ldots dw(\mathbf z_2)\\
       &- \frac{t}{n} \int_{\mathbb{R}^N}\ldots\int_{\mathbb{R}^N} h_{t/n}(\mathbf x,\mathbf z_1)W(\mathbf z_1)h_{t/n}(\mathbf z_1,\mathbf z_2) e^{-tV(\mathbf z_2)/n}\ldots 
       h_{t/n} (\mathbf z_{n-1}, \mathbf y)\, dw(\mathbf z_{n-1})\ldots dw(\mathbf z_1)\\
       & =:J_1^{[1]}(\mathbf x,\mathbf y) - J_2^{[1]}(\mathbf x,\mathbf y).
    \end{split} 
\end{equation*}
Observe that by Lemma \ref{lem:iter_lip}, we have 
\begin{equation*}
    \begin{split}
        |J_2^{[1]} (\mathbf x,\mathbf y)-J_2^{[1]}(\mathbf x',\mathbf y)|
   & \leq \frac{t\|V\|_\infty}{n}\int_{\mathbb{R}^N} | h_{t/n}(\mathbf x,\mathbf z_1)-h_{t/n}(\mathbf x',\mathbf z_1)| h_{(n-1)t/n}(\mathbf z_1,\mathbf y)\, \, dw(\mathbf z_1)\\
   &\leq C \frac{t\| V\|_\infty }{n} \frac{\| \mathbf x-\mathbf x'\|}{\sqrt{t/n}}\Big( h_{ct/n+(n-1)t/n}(\mathbf{x},\mathbf{y})+  h_{ct/n+(n-1)t/n}(\mathbf{x}',\mathbf{y})\Big)\\
   &\leq   C \frac{t\| V\|_\infty }{\sqrt{n}} \frac{\| \mathbf x-\mathbf x'\|}{\sqrt{t}}\Big( h_{ct/n+(n-1)t/n}(\mathbf{x},\mathbf{y})+  h_{ct/n+(n-1)t/n}(\mathbf{x}',\mathbf{y})\Big).
    \end{split}
\end{equation*}
It follows from~\eqref{heat:Rosler} that $h_{ct/n+(n-1)t}(\mathbf{x},\mathbf{y}) \leq C'h_{ct}(\mathbf{x},\mathbf{y})$. Hence, in the first step we have got 
\begin{equation}
    \begin{split}
        |q_{n,t}(\mathbf x,\mathbf y)- q_{n,t}(\mathbf x',\mathbf y)|&\leq |J_1^{[1]}(\mathbf x,\mathbf y)-J_1^{[1]}(\mathbf x',\mathbf y)|\\
        &+ C \frac{t\| V\|_\infty }{\sqrt{n}} \frac{\| \mathbf x-\mathbf x'\|}{\sqrt{t}}\Big( h_{ct}(\mathbf{x},\mathbf{y}) + h_{ct}(\mathbf{x}',\mathbf{y})\Big).
    \end{split}
\end{equation}

In the second step, we deal with $J_1^{[1]}(\cdot,\cdot)$. 
For $\mathbf{z}_2 \in \mathbb{R}^N$ we write  $\exp(-tV(\mathbf z_2)/n)=1-tW(\mathbf z_2)/n$,  
where $|W(\mathbf z_2)|\leq \| V\|_{\infty}$ and plug to the formula for $J_1^{[1]}(\cdot,\cdot)$. 
Thus 
\begin{equation*}
    \begin{split}
       & J_1^{[1]}(\mathbf{x},\mathbf{y})\\
       &=\int_{\mathbb{R}^N}\ldots\int_{\mathbb{R}^N} h_{2t/n}(\mathbf x,\mathbf z_2) (1-\frac{t}{n} W(\mathbf z_2))h_{t/n}(\mathbf z_2,\mathbf z_3)e^{-tV(\mathbf z_3)/n}\ldots 
       h_{t/n} (\mathbf z_{n-1}, \mathbf y)\, dw(\mathbf z_{n-1})\ldots\;dw(\mathbf z_2)\\
       &=\int_{\mathbb{R}^N}\ldots\int_{\mathbb{R}^N} h_{3t/n}(\mathbf x,\mathbf z_3) e^{-tV(\mathbf z_3)/n}h_{t/n}(\mathbf z_3,\mathbf z_4)e^{-tV(\mathbf z_4)/n}\ldots 
       h_{t/n} (\mathbf z_{n-1}, \mathbf y)\, dw(\mathbf z_{n-1})\ldots\;dw(\mathbf z_3)\\
       &-\frac{t}{n} \int_{\mathbb{R}^N}\ldots\int_{\mathbb{R}^N} h_{2t/n}(\mathbf x,\mathbf z_2) W(\mathbf z_2)h_{t/n}(\mathbf z_2,\mathbf z_3)e^{-tV(\mathbf z_3)/n}\ldots 
       h_{t/n} (\mathbf z_{n-1}, \mathbf y)\, dw(\mathbf z_{n-1})\ldots\;dw(\mathbf z_2)\\
       &=:J^{[2]}_1(\mathbf x,\mathbf y)-J^{[2]}_2(\mathbf x,\mathbf y). 
    \end{split}
\end{equation*}
Further, by Lemma \ref{lem:iter_lip},  
\begin{equation*}
    \begin{split}
        |J_2^{[2]} (\mathbf x,\mathbf y)-J_2^{[2]}(\mathbf x',\mathbf y)|
   & \leq \frac{Ct\|V\|_\infty}{n}\int_{\mathbb{R}^N} | h_{2t/n}(\mathbf x,\mathbf z_2)-h_{2t/n}(\mathbf x',\mathbf z_2)| h_{(n-2)t/n}(\mathbf z_2,\mathbf y)\, \, dw(\mathbf z_2)\\
   &\leq \frac{Ct\| V\|_\infty }{n} \frac{\| \mathbf x-\mathbf x'\|}{\sqrt{2t/n}}  C \Big(h_{2ct/n +(n-2)t/n}(\mathbf x,\mathbf y)
+   h_{2ct/n +(n-2)t/n}(\mathbf x,\mathbf y)\Big)\\
&\leq   \frac{Ct\| V\|_\infty }{\sqrt{2n}} \frac{\| \mathbf x-\mathbf x'\|}{\sqrt{t}}   \Big(h_{ct}(\mathbf x,\mathbf y)
+   h_{ct}(\mathbf x,\mathbf y)\Big).
    \end{split}
\end{equation*}
Thus, at the end of the second step, we have 
\begin{equation}
    \begin{split}
        |q_{n,t}(\mathbf x,\mathbf y)- q_{n,t}(\mathbf x',\mathbf y)|&\leq |J_1^{[2]}(\mathbf x,\mathbf y)-J_1^{[2]}(\mathbf x,\mathbf y)|\\
        &+ C  t\| V\|_\infty \frac{\|\mathbf x-\mathbf x'\|}{\sqrt{t}}\Big(\frac{1}{\sqrt{n}} + \frac{1}{\sqrt{2n}}\Big)\Big(h_{ct}(\mathbf x,\mathbf y)
+   h_{ct}(\mathbf x,\mathbf y)\Big). 
    \end{split}
\end{equation}

We continue this procedure, obtaining at the of the $m$-th step, 
$1\leq m\leq n-1$, the bound 

\begin{equation}
    \begin{split}
        |q_{n,t}(\mathbf x,\mathbf y)- q_{n,t}(\mathbf x',\mathbf y)|&\leq |J_1^{[m]}(\mathbf x,\mathbf y)-J_1^{[m]}(\mathbf x',\mathbf y)|\\
        &+  C  t\| V\|_\infty \frac{\|\mathbf x-\mathbf x'\|}{\sqrt{t}}\Big(\sum_{\ell=1}^m\frac{1}{\sqrt{\ell n}}\Big)\Big(h_{ct}(\mathbf x,\mathbf y)
+   h_{ct}(\mathbf x,\mathbf y)\Big) , 
    \end{split}
\end{equation}
where 
\begin{align*}
    J_1^{[m]}(\mathbf x,\mathbf y):= & \int_{\mathbb{R}^N} \ldots\int_{\mathbb{R}^N} h_{(m+1)t/n}(\mathbf{x},\mathbf{z}_{m+1})e^{-tV(\mathbf z_{m+1})}h_{t/n}(\mathbf z_{m+1},\mathbf z_{m+2})e^{-tV(\mathbf z_{m+2})/n}\times\dots\\   &\times h_{t/n}(\mathbf z_{n-1},\mathbf y)\, dw(\mathbf z_{n-1})\dots \;dw(\mathbf z_{m+1}).
\end{align*} 
Finally, we end up with the bound 

\begin{equation*}
    \begin{split}
        |q_{n,t}(\mathbf x,\mathbf y)- q_{n,t}(\mathbf x',\mathbf y)|&\leq |h_t(\mathbf x,\mathbf y)-h_t(\mathbf x',\mathbf y)|\\
        &+  C  t\| V\|_\infty \frac{\|\mathbf x-\mathbf x'\|}{\sqrt{t}}\Big(\sum_{\ell=1}^{n-1}\frac{1}{\sqrt{\ell n}}\Big)\Big(h_{ct}(\mathbf x,\mathbf y)
+   h_{ct}(\mathbf x',\mathbf y)\Big)\\
&\leq |h_t(\mathbf x,\mathbf y)-h_t(\mathbf x',\mathbf y)|
         +  C  t\| V\|_\infty \frac{\|\mathbf x-\mathbf x'\|}{\sqrt{t}}\Big(h_{ct}(\mathbf x,\mathbf y)
+   h_{ct}(\mathbf x',\mathbf y)\Big) \\
&\leq C(1+t\|V\|_\infty) \frac{\|\mathbf x-\mathbf x'\|}{\sqrt{t}}\Big(h_{ct}(\mathbf x,\mathbf y)
+   h_{ct}(\mathbf x',\mathbf y)\Big).\\
    \end{split}
\end{equation*}
By the same argument, 
\begin{equation}
    \begin{split}
        |q_{n,t}(\mathbf x,\mathbf y)- q_{n,t}(\mathbf x,\mathbf y')|
        &\leq C(1+t\|V\|_\infty) \frac{\|\mathbf y-\mathbf y'\|}{\sqrt{t}}\Big(h_{ct}(\mathbf x,\mathbf y)
+   h_{ct}(\mathbf x,\mathbf y')\Big).\\
    \end{split}
\end{equation}
Recall that $0<t\leq \| V\|_{\infty}^{-1}$. By the Arzel\'a-Ascoli theorem, there is a subsequence $\{n_j\}_{j \in \mathbb{N}}$ such that 
$\{q_{n_j,t}(\mathbf x,\mathbf y)\}_{j \in \mathbb{N}}$ converges uniformly on all compact sets of $\mathbb R^N\times \mathbb R^N$ to a continuous function $(\mathbf{x},\mathbf{y}) \longmapsto k^{\{V\}}_t(\mathbf x,\mathbf y)$,  which satisfies: 
\begin{equation}
    \begin{split}
      & 0\leq  k^{\{V\}}_t(\mathbf x,\mathbf y)\leq h_t(\mathbf x,\mathbf y),\\
       & |k^{\{V\}}_t(\mathbf x,\mathbf y)-k^{\{V\}}_t(\mathbf x',\mathbf y')|\leq C\frac{\|\mathbf x-\mathbf x'\|+\|\mathbf y-\mathbf y'\|}{\sqrt{t}}h(ct,\mathbf x,\mathbf y,\mathbf x',\mathbf y') .
    \end{split}
\end{equation}
Observe that the sequence $\{Q_{n_j,t}(\mathbf x,\mathbf y)\}_{j \in \mathbb{N}}$ converges uniformly on compact subsets of $\mathbb R^N\times\mathbb R^N$ to $k^{\{V\}}_t(\mathbf x,\mathbf y)$ as well. 
By the product formula~\eqref{eq:product-formula}, for all  $f\in L^2(dw)$,  we have 
\begin{equation}\label{eq:product_formula2} e^{t(\Delta_k-V)}f(\mathbf x)=\lim_{n\to\infty} \int_{\mathbb{R}^N} Q_{t,n}(\mathbf x,\mathbf y)f(\mathbf y)\, dw(\mathbf y)
\end{equation}
with the convergence in the $L^2(dw(\mathbf x))$-norm. Recall that $Q_{t,n}(\mathbf x,\mathbf y)\leq h_t(\mathbf x,\mathbf y)$ (see \eqref{eq:q_n_t} and \eqref{eq:q_n_t_1}). Thus, by the Lebesgue dominated convergence theorem, for all $\mathbf x\in\mathbb R^N$, one has 
$$ \lim_{j\to\infty} \int_{\mathbb{R}^N} Q_{t,n_j}(\mathbf x,\mathbf y)f(\mathbf y)\, dw(\mathbf y)=\int_{\mathbb{R}^N} k^{\{V\}}_t(\mathbf x,\mathbf y)f(\mathbf y)\, dw(\mathbf y).$$ 
Thus \eqref{eq:semigroup_form} is established. 
 
We now turn to prove \eqref{eq:time_der}. The operator  $-\Delta_k+V$ is non-negative and self-adjoint on $L^2(dw)$. Thus, by the spectral theorem, the mapping $(0,\infty)\ni t\mapsto e^{t(\Delta_k-V)}\in \mathcal L(L^2(dw))$ is a smooth function, and for  any $m \in \mathbb{N}$ there is $C_m>0$ such that for all measurable and bounded $V\geq 0$ and $t>0$ we have
\begin{align*}
    \left\| \frac{d^m}{dt^m} e^{t(\Delta_k-V)}\right\|_{\mathcal L(L^2(dw))}\leq C_mt^{-m}.    
\end{align*}
Here $\mathcal L(L^2(dw))$ denotes the Banach space of bounded linear operators on $L^2(dw)$. Thus, 
$$ \Big|\frac{d^m}{dt^m}\langle e^{t(\Delta_k-V)} f,g\rangle_{L^2(dw)}\Big|\leq C_m t^{-m}\| f\|_{L^2(dw)}\|g\|_{L^2(dw)}. $$ For $t>0$, set $t_0=t/4$. Then for fixed $\mathbf x,\mathbf y\in\mathbb R^N$, we have 
$$k^{\{V\}}_t(\mathbf x,\mathbf y)=\Big\langle e^{(t-2t_0)(\Delta_k -V)}k_{t_0}(\cdot,\mathbf y), k_{t_0}(\mathbf x,\cdot)\Big\rangle_{L^2(dw)}. $$
Hence \eqref{eq:time_der} follows, since $\| k_{t_0}(\cdot, \mathbf y)\|_{L^2(dw)}\leq C w(B(\mathbf y,\sqrt{t}))^{-1/2}$, by Theorem~\ref{teo:1}. 
\end{proof}

\begin{corollary}\label{coro:mono}
Assume that $V_1,V_2 \colon \mathbb{R}^N \longmapsto [0,\infty)$, $V_1,V_2$ are bounded, and $V_1(\mathbf{y}) \leq V_2(\mathbf{y})$ for all $\mathbf{y} \in \mathbb{R}^N$. Then for all $\mathbf{x},\mathbf{y} \in \mathbb{R}^N$ and $t>0$ we have 
\begin{align*}
    k_{t}^{\{V_2\}}(\mathbf{x},\mathbf{y}) \leq k_{t}^{\{V_1\}}(\mathbf{x},\mathbf{y}).
\end{align*}
\end{corollary}

\begin{proof}
It is enough to note that the assumption $V_1(\mathbf{y}) \leq V_2(\mathbf{y})$ implies
\begin{align*}
    &k_{t}^{\{V_2\}}(\mathbf{x},\mathbf{y})\\&=\lim_{n \to \infty} \int_{\mathbb{R}^N}\ldots\int_{\mathbb{R}^N} h_{t/n}(\mathbf x,\mathbf z_1)e^{-tV_2(\mathbf z_1)/n}h_{t/n}(\mathbf z_1,\mathbf z_2)e^{-tV_2(\mathbf z_2)/n}\ldots 
       h_{t/n} (\mathbf z_{n-1}, \mathbf y)\, dw(\mathbf z_{n-1})\ldots dw(\mathbf z_1)\\&\leq \lim_{n \to \infty} \int_{\mathbb{R}^N}\ldots\,\int_{\mathbb{R}^N} h_{t/n}(\mathbf x,\mathbf z_1)e^{-tV_1(\mathbf z_1)/n}h_{t/n}(\mathbf z_1,\mathbf z_2)e^{-tV_1(\mathbf z_2)/n}\ldots \,
       h_{t/n} (\mathbf z_{n-1}, \mathbf y)\, dw(\mathbf z_{n-1})\ldots dw(\mathbf z_1)\\&=k_{t}^{\{V_1\}}(\mathbf{x},\mathbf{y}).
\end{align*}
\end{proof}

\section{Upper bounds for Schr\"odinger semigroups with non-negative potentials} 
\begin{theorem}\label{teo:kernel_etL}
Assume that $V\colon \mathbb{R}^N \longmapsto [0,\infty)$, $V\in L^1_{\rm loc}(dw)$. Let $V_n=\min(V,n)$ and $L_n=-\Delta_k+V_n$, $n \in \mathbb{N}$. Then, for all $\mathbf{x},\mathbf{y} \in \mathbb{R}^N$ and $t>0$ the sequence $\{k_t^{\{V_n\}}(\mathbf x,\mathbf y)\}_{n \in \mathbb{N}}$ converges monotonically to the kernel of the semigroup $\{e^{-tL}\}_{t \geq 0}$, that is, for all $\mathbf{x},\mathbf{y} \in \mathbb{R}^N$ and $t>0$ we have 
$$ \lim_{n\to\infty} k_t^{\{V_n\}}(\mathbf x,\mathbf y)=k^{\{V\}}_t(\mathbf x,\mathbf y)$$ 
and 
\begin{equation}\label{eq:form_eL} e^{-tL}f(\mathbf x)=\int_{\mathbb{R}^N} k^{\{V\}}_t(\mathbf x,\mathbf y) f(\mathbf y)\, dw(\mathbf y). 
\end{equation}
Moreover, for any $m \in \mathbb{N}$ there is a constant $C_m>0$ such that for all  $(\mathbf x,\mathbf y)\in\mathbb R^N\times\mathbb R^N$ the function $(0,\infty)\ni t\mapsto k^{\{V\}}_t(\mathbf x,\mathbf y)$ is smooth and 
\begin{equation}\label{eq:analitic_der} \Big|\frac{d^m}{dt^m}k^{\{V\}}_t(\mathbf x,\mathbf y)\Big|\leq C_mt^{-m} w(B(\mathbf x,\sqrt{t}))^{-1/2} w(B(\mathbf y,\sqrt{t}))^{-1/2}.
\end{equation}
\end{theorem}
\begin{proof}
By the results of the previous section (see Theorem~\ref{teo:continuous_k} and Corollary~\ref{coro:mono}) the kernels $\{k_t^{\{V_n\}}(\mathbf x,\mathbf y)\}_{n \in \mathbb{N}}$ of the semigroups $\{e^{-tL_n}\}_{t \geq 0}$, form a monotonic family of continuous functions of $(t,\mathbf x,\mathbf y)$, that is, 
\begin{align*}
     0\leq k_t^{\{V_{n+1}\}}(\mathbf x,\mathbf y)\leq  k_t^{\{V_n\}}(\mathbf x,\mathbf y)\leq h_t(\mathbf x,\mathbf y).
\end{align*}
Hence, for all $(\mathbf{x},\mathbf{y}) \in \mathbb{R}^N \times \mathbb{R}^N$ and $t>0$ the limit $\lim_{n\to\infty}k_t^{\{V_n\}}(\mathbf x,\mathbf y)$ exists and defines a kernel $k^{\{V\}}_t(\mathbf x,\mathbf y)\leq h_t(\mathbf x,\mathbf y)$. Moreover, applying the Arzelà–Ascoli theorem, we deduce     \eqref{eq:analitic_der} from the inequalities
\begin{align*}
    \Big|\frac{d^m}{dt^m}k_t^{\{V_n\}}(\mathbf x,\mathbf y)\Big|\leq C_mt^{-m} w(B(\mathbf x,\sqrt{t}))^{-1/2} w(B(\mathbf y,\sqrt{t}))^{-1/2},
\end{align*}
which hold for $k_t^{\{V_n\}}(\mathbf x,\mathbf y)$ thanks to Theorem \ref{teo:continuous_k} (see~\eqref{eq:time_der}). Further, by the Lebesgue dominated convergence theorem,  for each  $\mathbf x\in\mathbb R^N$ and all  $f\in L^2(dw)$, the limit  
$$ \lim_{n\to\infty} e^{-tL_n}f(\mathbf x)=\lim_{n\to\infty} \int_{\mathbb{R}^N} k_t^{\{V_n\}}(\mathbf x,\mathbf y)f(\mathbf y)\, dw(\mathbf y)$$ 
exists and defines a bounded functional such that 
$$ \lim_{n\to\infty} e^{-tL_n}f(\mathbf x)=\int_{\mathbb{R}^N} k^{\{V\}}_t(\mathbf x,\mathbf y)f(\mathbf y)\, dw(\mathbf y).$$ 
On the other hand, by \eqref{eq:basic_convergence} for each $f\in L^2(dw)$, the sequence $\{e^{-tL_n}f\}_{n \in \mathbb{N}}$ converges in the $L^2(dw)$-norm to $e^{-tL}f$, hence \eqref{eq:form_eL} follows.
\end{proof}

\begin{corollary}\label{coro:mono_1}
Assume that $V_1,V_2 \colon \mathbb{R}^N \longmapsto [0,\infty)$, $V_1,V_2 \in L^1_{{\rm loc}}(dw)$ and $V_1(\mathbf{y}) \leq V_2(\mathbf{y})$ for all $\mathbf{y} \in \mathbb{R}^N$. Then for all $\mathbf{x},\mathbf{y} \in \mathbb{R}^N$ and $t>0$ we have 
\begin{align*}
    k_{t}^{\{V_2\}}(\mathbf{x},\mathbf{y}) \leq k_{t}^{\{V_1\}}(\mathbf{x},\mathbf{y}).
\end{align*}
\end{corollary}

\begin{proof}
It is a consequence of Corollary~\ref{coro:mono} and Theorem~\ref{teo:kernel_etL}.
\end{proof}

\section{The Feynman-Kac formula}\label{sec:F-K}  In this section we  elaborate  the Feynman-Kac formula for Dunkl Schr\"odinger operators with continuous bounded potentials. Our approach is standard and uses the product formula \eqref{eq:product-formula} (see also \eqref{eq:product_formula2}). For the reader convenience, we provide some details.  Then the Feynman--Kac formula will be used in proving the implication \eqref{item:potential_Green_bounded} $\implies$ \eqref{item:h_ct_leq_k_t} of Theorem \ref{teo:main}.  

Let  $I\subset \mathbb R$ be an interval. Recall that a function $I\ni t\longmapsto X_t\in \mathbb R^N$  is said to be {\it c\`adl\`ag} if it is right continuous, and it has  left  limits.

\begin{proposition}\label{prop:integral_cadlag}
Assume that $f \colon [a,b]\to \mathbb R$ is a bounded c\`adl\`ag function. Then 
$$ \lim_{n\to\infty} \frac{b-a}{n}\sum_{k=0}^{n-1} f\left(a+\frac{k(b-a)}{n}\right)=\int_a^b f(t)\, dt. $$
The right-hand side of the formula above is understood as the Lebesgue integral. 
\end{proposition}
The proposition can be proved by standard arguments. For the completeness we elaborate it in Appendix~\ref{sec:app}.

Let $(X_t,\Omega, \mathbb P^{\mathbf x})$, $X_t:\Omega\to \mathbb R^N$,  be a Dunkl process associated with the transition probabilities
\begin{align*}
    P_t(\mathbf x,E)=\int_E h_t(\mathbf x,\mathbf y)\, dw(\mathbf y),    
\end{align*}
that is, a Markov process with  c\`adl\`ag realizations $[0,\infty)\ni t\mapsto X_t(\omega)$ satisfying  
\begin{equation}
    \begin{split}
       & \mathbb P^\mathbf{x}\{\omega\in\Omega: X_{t_1}\in E_1, X_{t_2}\in E_2,\ldots,X_{t_n}\in E_n\}\\
       &=\int_{E_1}\int_{E_2}\ldots\int_{E_n} h_{t_1}(\mathbf{x},\mathbf{x}_1)h_{t_2-t_1}(\mathbf{x}_1,\mathbf{x}_2)\ldots h_{t_n-t_{n-1}}(\mathbf{x}_{n-1},\mathbf{x}_n)\, dw(\mathbf{x}_n)\, dw(\mathbf{x}_{n-1})\ldots dw(\mathbf{x}_1).
    \end{split}
\end{equation}
for any finite sequence $0<t_1<t_t<\ldots<t_n$ and any  measurable sets $E_1,E_2,\ldots,E_n\subseteq \mathbb R^N$ (see R\"osler-Voit~\cite{RV}). 
The formula implies that for a reasonable measurable function $F$ defined on $(\mathbb R^N)^n$  one has 
\begin{equation}\label{eq:FK_pre}
    \begin{split}
       & \mathbb E^\mathbf{x} (F(X_{t_1},X_{t_2},\ldots,X_{t_n}))\\
       &=\int_{(\mathbb R^N)^n} F(\mathbf{x}_1,\mathbf{x}_2,\ldots,\mathbf{x}_n)h_{t_1}(\mathbf{x},\mathbf{x}_1)h_{t_2-t_1}(\mathbf{x}_1,\mathbf{x}_2)\ldots h_{t_n-t_{n-1}}(\mathbf{x}_{n-1},\mathbf{x}_n)\, dw(\mathbf{x}_n)\, dw(\mathbf{x}_{n-1})\ldots dw(\mathbf{x}_1).
    \end{split}
\end{equation}
Assume  that $V\geq 0$ is a bounded continuous function. 
Let $Q_{n,t}(\mathbf x,\mathbf y)$ be as in the proof of Theorem \ref{teo:continuous_k} (see~\eqref{eq:q_n_t} and~\eqref{eq:q_n_t_1}). Let $f\in L^2(dw)$ and $t>0$. 
Putting $t_{k}:=\frac{k}{n}t$, $1 \leq k \leq n$,  and using ~\eqref{eq:FK_pre}, we have 
\begin{equation}
    \begin{split}
         ( H_{t/n}e^{-\frac{t}{n}V})^nf(\mathbf{x})&=\int_{\mathbb{R}^{N}} Q_{n,t}(\mathbf x,\mathbf y)f(\mathbf y)\, dw(\mathbf y)\\
        &= \int_{(\mathbb{R}^N)^n} h_{t/n} (\mathbf{x},\mathbf{x}_1)h_{t/n}(\mathbf{x}_1,\mathbf{x}_2)\ldots h_{t/n}(\mathbf{x}_{n-1}, \mathbf{x}_n)\\
        & \hskip1.5cm \times {\underbrace{e^{-\frac{t}{n} (V(\mathbf{x}_1)+V(\mathbf{x}_2)+\ldots+V(\mathbf{x}_n)) } f(\mathbf{x}_n)}_{F(\mathbf{x}_1,\mathbf{x}_2,\ldots,\mathbf{x}_n)}}\, dw(\mathbf{x}_n)\, dw(\mathbf{x}_{n-1})\, \ldots \, dw(\mathbf{x}_1)\\
        &=E^{\mathbf{x}} (F(X_{t_1},X_{t_2},\ldots,X_{t_n}))\\
        &=E^{\mathbf{x}} \Big[\exp\left(-\frac{t}{n}\sum_{k=1}^n V(X_{t_k})\right)f(X_{t}) \Big].
    \end{split}
\end{equation}
Recall that it was established in the proof of Theorem \ref{teo:continuous_k} that there is a subsequence $\{n_j\}_{j \in \mathbb{N}}$ such that the continuous functions  $Q_{n_j,t}(\mathbf{x},\mathbf{y})\leq h_t(\mathbf x,\mathbf y)$ converge uniformly on compact  subsets to $k^{\{V\}}_t(\mathbf x,\mathbf y)$. Hence 
taking into account integration of  c\`adl\`ag functions (see Proposition \ref{prop:integral_cadlag}), we obtain the following corollary. 

\begin{corollary}[Feynman-Kac formula]
Let $V\geq 0$ be a bounded continuous function. 
Then for $t>0$, $\mathbf x\in \mathbb R^N$ and all $f\in L^2(dw)$ we have  
\begin{equation}\label{eq:Feynman-Kac}
 e^{t(\Delta_k-V)}f(\mathbf x)=E^{\mathbf{x}} \Big[\exp\Big(-\int_0^t V(X_{s})\, ds\Big)f(X_{t}) \Big]   .
\end{equation}

\end{corollary}

\section{Schr\"odinger semigroups with Green bounded potentials} 

\subsection{Green bounded potentials}\label{sub:Green} In the proposition below we elaborate the equivalences stated in Remark \ref{remark:GG}. 
For a measurable function $V:\mathbb{R}^N \longmapsto [0,\infty)$ and $\mathbf{x} \in \mathbb{R}^N$, let
\begin{equation}\label{eq:mathbf_G}
    \mathbf G(V)(\mathbf x):=\int_0^\infty \int_{\mathbb R^N} w(B(\mathbf x,\sqrt{s}))^{-1} e^{-\| \mathbf x-\mathbf y\|^2/s}V(\mathbf y)\, dw(\mathbf y)\, ds, 
\end{equation}
\begin{equation}
     \mathbf G_1(V)(\mathbf x):= \int_0^\infty \int_{\mathbb R^N} h_s(\mathbf{x},\mathbf{y}) V(\mathbf y)\, dw(\mathbf y)\, ds,
\end{equation}
\begin{equation}\label{eq:mathcal_G}
    \mathcal G(V)(\mathbf x):= \int_0^\infty \int_{\mathbb R^N}w(B(\mathbf x,\sqrt{s}))^{-1} e^{-d (\mathbf x,\mathbf y)^2/s}V(\mathbf y)\, dw(\mathbf y)\, ds. 
    \end{equation}
    
\begin{proposition}\label{prop:Greens}
There are constants $C_1,C_2,C_3>0$ such that for all measurable non-negative functions $V:\mathbb{R}^N \longmapsto [0,\infty)$ one has 
\begin{equation}
    \|\mathbf G(V)\|_{L^\infty}\leq C_1 \| \mathbf G_1(V)\|_{L^\infty}\leq C_2\| \mathcal G(V)\|_{L^\infty} \leq C_3 \|\mathbf G(V)\|_{L^\infty}.
\end{equation}
\end{proposition}

\begin{proof} It follows from Theorem~\ref{teo:1} that there are constants $C,c>0$ such that for all $\mathbf{x},\mathbf{y} \in \mathbb{R}^N$ and $s>0$ we have 
\begin{align}\label{eq:simple_u_l}
    C^{-1}w(B(\mathbf x,\sqrt{s}))^{-1} e^{-c^{-1}\| \mathbf x-\mathbf y\|^2/s}\leq h_s(\mathbf x,\mathbf y)\leq C w(B(\mathbf x,\sqrt{s}))^{-1} e^{-c d(\mathbf x,\mathbf y)^2/s} 
\end{align}
(see also \cite[Theorems 4.1 and 4.4]{ADzH}).  Further, by the definition of $d(\mathbf x,\mathbf y)$ (see~\eqref{eq:d(x,y)}), 
\begin{equation}\label{eq:suma_G}
     w(B(\mathbf x,\sqrt{s}))^{-1} e^{-c d(\mathbf x,\mathbf y)^2/s}\leq \sum_{\sigma\in G}  w(B(\mathbf x,\sqrt{s}))^{-1} e^{-c \|\sigma (\mathbf x)-\mathbf y\|^2/s}. 
\end{equation}
The proposition is a direct consequence of the inequalities \eqref{eq:simple_u_l}, \eqref{eq:suma_G}, and the doubling property of the measure $dw$ (see~\eqref{eq:doubling}). 
\end{proof}

In order to establish Theorem~\ref{teo:main}, we prove the implications: ~\eqref{item:h_ct_leq_k_t}$\implies$~\eqref{item:int_k_geq_delta} (in Lemma~\ref{lem:e_implies_a}), then  
~\eqref{item:int_k_geq_delta}$\implies$~\eqref{item:potential_Green_bounded} (in Lemma~\ref{lem:a_implies_c}), and finally,  \eqref{item:potential_Green_bounded}$\implies$~\eqref{item:h_ct_leq_k_t} (in Subsection~\ref{subsection:main}). We prove
\eqref{item:potential_Green_bounded}$\implies$~\eqref{item:h_ct_leq_k_t} 
 in the separate subsection, because it is relatively more involving and it uses the  heat kernel estimates \eqref{eq:main_lower}, \eqref{eq:main_claim},  and the Feynman--Kac formula (see Section \ref{sec:F-K}). 

\subsection{Proofs of the implications\texorpdfstring{~\eqref{item:h_ct_leq_k_t}$\implies$~\eqref{item:int_k_geq_delta}}{(a)=>(b)} and\texorpdfstring{~\eqref{item:int_k_geq_delta}$\implies$~\eqref{item:potential_Green_bounded}}{(b)=>(c)}}
\begin{lemma}\label{lem:e_implies_a}
Assume that $V:\mathbb{R}^N \longmapsto [0,\infty)$, $V\in L^1_{\rm loc} (dw)$. Assume that there are constants $C,c>0$ such that for all $\mathbf{x},\mathbf{y} \in \mathbb{R}^N$ and $t>0$ we have 
\begin{equation}\label{eq:e_implies_a}
    h_{ct}(\mathbf{x},\mathbf{y})\leq C k_t^{\{V\}}(\mathbf{x},\mathbf{y}).
\end{equation}
Then there is a constant $\delta>0$ such for all $\mathbf{x} \in \mathbb{R}^N$ and $t>0$ we have
\begin{align*}
    \int_{\mathbb{R}^N} k_t^{\{V\}}(\mathbf x,\mathbf y)\, dw(\mathbf y) > \delta.
\end{align*}
\end{lemma}
\begin{proof}
It is enough to integrate~\eqref{eq:e_implies_a} with respect to $dw(\mathbf{y})$ and apply~\eqref{eq:probabilistic}.
\end{proof}

\begin{lemma}
Assume that $V:\mathbb{R}^N \longmapsto [0,\infty)$ is measurable and bounded. Then for all $t>0$ and $\mathbf{x},\mathbf{y} \in \mathbb{R}^N$ we have
\begin{equation}
    \label{eq:perturbation}
    h_t(\mathbf x,\mathbf y)=k^{\{V\}}_t(\mathbf x,\mathbf y)+\int_0^t \int_{\mathbb{R}^N} h_s(\mathbf x,\mathbf z)V(\mathbf z)k^{\{V\}}_{t-s}(\mathbf z,\mathbf y)\, dw(\mathbf z)\,ds.
\end{equation}
\end{lemma}

\begin{proof}
See Theorem~\ref{teo:AB}. 
\end{proof}

\begin{lemma}\label{lem:a_implies_c}
Assume that $V:\mathbb{R}^N \longmapsto [0,\infty)$, $V\in L^1_{\rm loc} (dw)$. Assume that there is $\delta>0$ such for all $\mathbf{x} \in \mathbb{R}^N$ and $t>0$ we have
\begin{equation}\label{eq:assumption_bc}
    \int_{\mathbb{R}^N} k_t^{\{V\}}(\mathbf x,\mathbf y)\, dw(\mathbf y) > \delta.
\end{equation}
Then $V$ is Green bounded. 
\end{lemma}

\begin{proof}
The proof is standard. Let $V_n=\min (V,n)$, $n \in \mathbb{N}$. Recall that   $k_t^{\{V_n\}}(\mathbf x,\mathbf y)\geq k_t^{\{V\}}(\mathbf x,\mathbf y)$ (see Corollaries~\ref{coro:mono} and~\ref{coro:mono_1}).  By the perturbation formula~\eqref{eq:perturbation} applied to $k_t^{\{V_n\}}(\mathbf{x},\mathbf{y})$,~\eqref{eq:probabilistic}, and the assumption~\eqref{eq:assumption_bc} we have
\begin{equation*}
    \begin{split}
        1 &\geq \int_{\mathbb R^N}\int_0^t \int_{\mathbb R^N} h_{s} (\mathbf x,\mathbf z)V_n(\mathbf z)k_{t-s}^{\{V_n\}}(\mathbf z,\mathbf y)\, dw(\mathbf z)\, ds\, dw(\mathbf y) 
        \geq \delta \int_0^t \int_{\mathbb R^N}h_s(\mathbf x,\mathbf z) V_n(\mathbf z)\, dw(\mathbf z)\, ds\\
    \end{split}
\end{equation*}
with $\delta$ independent of $n \in \mathbb{N}$. Letting $t\to\infty$, we obtain the  bound  $\|\mathbf G_1(V_n)\|_{L^\infty}\leq \delta^{-1}$.    Now, letting $n\to\infty$, we get the lemma by applying the Lebesgue monotone convergence theorem and Proposition~\ref{prop:Greens}. 
\end{proof}

\subsection{Implication\texorpdfstring{~\eqref{item:potential_Green_bounded}$\implies$~\eqref{item:h_ct_leq_k_t}}{(c)=>(a)}}\label{subsection:main}
In order to prove the implication we adapt to the Dunkl setting general patterns of  proofs of lower bounds  for the classical Schr\"odinger operators or Bessel-Schr\"odinger operators (see~\cite{Sem},~\cite{Dziub-Preisner}).  Thus, first we prove the lower bounds in the case of continuous and bounded $V$ with the property  $\| \mathcal G(V)\|$ being small enough. Then we extend the lower estimates to all non-negative Green bounded potentials $V$. The main difficulties we face concern the fact that the upper and lower estimates of the Dunkl heat kernel $h_t(\mathbf x,\mathbf y)$ have rather complex forms which involve both - the orbit distance $d(\mathbf x,\mathbf y) $  and the Euclidean distances  $\| \mathbf x-\sigma_{(\alpha_1,\alpha_2,\ldots,\alpha_j)}(\mathbf y)\|$ contained in the definition of the function $\Lambda$ (see~\eqref{eq:Lambda_def}). To this end we need a preparation. 

For $\mathbf{x},\mathbf{y} \in \mathbb{R}^N$ and $t>0$ we set
\begin{equation}\label{eq:G_2_variables}
    \mathcal G_t(\mathbf x,\mathbf y):=w(B(\mathbf y,\sqrt{t}))^{-1} e^{-d(\mathbf x,\mathbf y)^2/t}.
\end{equation}

Let us begin with a proposition concerning the properties of the generalized heat kernel. In its proof, the specific generalized heat kernel bounds from Theorem~\ref{teo:1} are utilized.

\begin{proposition}\label{propo:glue_together}
There are constants $C_1>0$, $c_1>1$, such that for all $0<s \leq t/2$ and all $\mathbf{x},\mathbf{y},\mathbf{z}\in \mathbb R^N$ one has 
\begin{equation}
    h_{t-s}(\mathbf{x},\mathbf{z})h_s(\mathbf{z},\mathbf{y})\leq C_{1}  h_{{c_{1}} t}(\mathbf{x},\mathbf{y})\mathcal G_{c_1 s}(\mathbf{z},\mathbf{y}).
\end{equation}
\end{proposition}
\begin{proof}
Let $c_0>1$. By Lemma~\ref{lem:small_deformation}, for all $\mathbf{z}' \in \mathbb{R}^N$ such that $\|\mathbf{z}-\mathbf{z}'\| \leq \sqrt{s} \leq \sqrt{t-s}$ we have
\begin{equation}\label{eq:uni_1}
    h_{t-s}(\mathbf{x},\mathbf{z})h_s(\mathbf{z},\mathbf{y}) \leq C_0^2h_{c_0(t-s)}(\mathbf{x},\mathbf{z}')h_{c_0s}(\mathbf{z}',\mathbf{y}).
\end{equation}
Note that 
$$ e^{-c_u\frac{d(\mathbf z', \mathbf y)^2}{c_0s}}\leq C  e^{-c_u\frac{d(\mathbf z', \mathbf y)^2}{2c_0s}} e^{-c_u\frac{d(\mathbf z, \mathbf y)^2}{4c_0s}} \quad \text{for } \|\mathbf z-\mathbf z'\|\leq \sqrt{s}.$$
Hence, applying \eqref{eq:main_claim},  ~\eqref{eq:Lambda_2t}, the doubling property of $dw$  (see~\eqref{eq:doubling}), and Theorem~\ref{teo:1}, we get
\begin{equation}\label{eq:uni_2}
    \begin{split}
        &h_{c_0(t-s)}(\mathbf{x},\mathbf{z}')h_{c_0s}(\mathbf{z}',\mathbf{y}) \\&\leq Cw(B(\mathbf{z}',\sqrt{c_0(t-s)}))^{-1}w(B(\mathbf{z}',\sqrt{c_0s}))^{-1}\Lambda(\mathbf{x},\mathbf{z}',c_0(t-s))\Lambda(\mathbf{z}',\mathbf{y},c_0s)e^{-c_{u}\frac{d(\mathbf{x},\mathbf{z}')^2}{c_0(t-s)}-c_{u}\frac{d(\mathbf{z}',\mathbf{y})^2}{c_0s}}\\&\leq C \left(w(B(\mathbf{z}',\sqrt{t-s}))^{-1}\Lambda(\mathbf{x},\mathbf{z}',t-s)\Lambda(\mathbf{z}',\mathbf{y},s)e^{-c_u\frac{d(\mathbf{x},\mathbf{z}')^2}{2c_0(t-s)}-c_u\frac{d(\mathbf{z}',\mathbf{y})^2}{2c_0s}}\right)\left(w(B(\mathbf{z},\sqrt{s}))^{-1}e^{-c_{u}\frac{d(\mathbf{y},\mathbf{z})^{2}}{{4c_0s}}}\right)\\&\leq C w(B(\mathbf{z},\sqrt{s})) \mathcal{G}_{c_1s}(\mathbf{y},\mathbf{z})h_{c_1(t-s)}(\mathbf{x},\mathbf{z}')h_{c_1s}(\mathbf{z}',\mathbf{y}).
    \end{split}
\end{equation}
Since the estimates~\eqref{eq:uni_1} and~\eqref{eq:uni_2} are given uniformly on $\mathbf{z}' \in B(\mathbf{z}, \sqrt{s})$, taking their mean over the ball $B(\mathbf{z}, \sqrt{s})$ we get
\begin{align*}
    h_{t-s}(\mathbf{x},\mathbf{z})h_s(\mathbf{z},\mathbf{y}) &\leq C\mathcal{G}_{c_1s}(\mathbf{y},\mathbf{z})\int_{B(\mathbf{z},\sqrt{s})}h_{c_1(t-s)}(\mathbf{x},\mathbf{z}')h_{c_1 s}(\mathbf{z}',\mathbf{y})\,dw(\mathbf{z}') \\& \leq C\mathcal{G}_{c_1 s}(\mathbf{y},\mathbf{z})\int_{\mathbb{R}^N}h_{c_1(t-s)}(\mathbf{x},\mathbf{z}')h_{c_1s}(\mathbf{z}',\mathbf{y})\,dw(\mathbf{z}')\\&=C\mathcal{G}_{c_1 s}(\mathbf{y},\mathbf{z}) h_{c_1t}(\mathbf{x},\mathbf{y}),
\end{align*}
where is the last step we have used the semigroup property of $h_t(\cdot,\cdot)$.
\end{proof}

The following corollary is an consequence of Proposition~\ref{propo:glue_together}. 

\begin{corollary}\label{coro:pert_est}
Assume that $V:\mathbb{R}^N \longmapsto [0,\infty)$ is continuous, bounded, and Green bounded. Then there are constants $C_2,c_2>0$ such that for all $\mathbf{x},\mathbf{y}\in \mathbb{R}^N$ we have
\begin{align*}
    \int_0^t\int_{\mathbb{R}^N} h_{t-s}(\mathbf{x},\mathbf{z})V(\mathbf{z})k^{\{V\}}_s(\mathbf{z},\mathbf{y})\, dw(\mathbf{z})\, ds \leq C_2\|\mathcal G(V)\|_{L^\infty}h_{{c_2}t}(\mathbf{x},\mathbf{y}).
\end{align*}
\end{corollary}

\begin{proof}
By~\eqref{eq:k_very_basic} we have
\begin{align*}
   \int_0^t\int_{\mathbb{R}^N} h_{t-s}(\mathbf{x},\mathbf{z})V(\mathbf{z})k^{\{V\}}_s(\mathbf{z},\mathbf{y})\, dw(\mathbf{z})\, ds &\leq \int_0^t\int_{\mathbb{R}^N} h_{t-s}(\mathbf{x},\mathbf{z})V(\mathbf{z})h_s(\mathbf{z},\mathbf{y})\, dw(\mathbf{z})\, ds  \\&=\int_0^{t/2}\int_{\mathbb{R}^N}\cdots+\int_{t/2}^{t}\int_{\mathbb{R}^N}\cdots=:I_1+I_2.
\end{align*}
We will estimate  $I_1$; the case of $I_2$ can be reduced to the case of $I_1$ by the change of variables. By Proposition~\ref{propo:glue_together} we get
\begin{align*}
    I_1 \leq C_1 h_{c_1t}(\mathbf{x},\mathbf{y})\int_0^t\int_{\mathbb{R}^N}\mathcal{G}_{c_1 s}(\mathbf{z},\mathbf{y})V(\mathbf{z})\, dw(\mathbf{z})\, ds.
\end{align*}
Finally, by the change of variables $s_1:=c_1 s$ we get
\begin{align*}
    \int_0^t\int_{\mathbb{R}^N}\mathcal{G}_{c_1s}(\mathbf{z},\mathbf{y})V(\mathbf{z})\, dw(\mathbf{z})\, ds \leq C\|\mathcal{G}(V)\|_{L^{\infty}},
\end{align*}
which finishes the proof.
\end{proof}

\begin{corollary}\label{eq:lower_small}
Assume that $V:\mathbb{R}^N \longmapsto [0,\infty)$ is bounded, continuous, and Green bounded. Let $c_3>0$. There are $\widetilde{c}_3,\varepsilon>0$ such that if $\|\mathcal G(V)\|_{L^\infty}<\varepsilon$, then for all $\mathbf{x},\mathbf{y} \in \mathbb{R}^N$ and $t>0$ such that $d(\mathbf{x},\mathbf{y})<c_3\sqrt{t}$ one has
\begin{equation}\label{eq:small_t_fin}
    k^{\{V\}}_t(\mathbf{x},\mathbf{y}) \geq \widetilde{c}_3 h_t(\mathbf{x},\mathbf{y}).
\end{equation}
\end{corollary}

\begin{proof} From the perturbation formula~\eqref{eq:perturbation} we get
\begin{align*}
    h_t(\mathbf{x},\mathbf{y})-k^{\{V\}}_{t}(\mathbf{x},\mathbf{y})=\int_0^{t}\int_{\mathbb{R}^N}h_{t-s}(\mathbf{x},\mathbf{z})V(\mathbf{z})k^{\{V\}}_{s}(\mathbf{z},\mathbf{y})\,dw(\mathbf{z})\,ds.
\end{align*}
Hence from  Corollary~\ref{coro:pert_est},~\eqref{eq:main_claim}, and~\eqref{eq:main_lower}, we deduce  that 
\begin{equation}\label{eq:pert_app}
    \begin{split}
        k^{\{V\}}_t(\mathbf{x},\mathbf{y})&\geq h_t(\mathbf{x},\mathbf{y})-C_2\| \mathcal G(V)\|_{L^\infty} h_{c_2t}(\mathbf{x},\mathbf{y})\\
       & \geq C_{\rm l} w(B(\mathbf{x},\sqrt{t}))^{-1} e^{-c_{\rm l}\frac{d(\mathbf{x},\mathbf{y})^2}{t}}\Lambda(\mathbf{x},\mathbf{y},t) \\
       &\ \ - C_{\rm u} C_2 \| \mathcal G(V)\|_{L^\infty} w(B(\mathbf{x},\sqrt{c_2t}))^{-1}\Lambda(\mathbf{x},\mathbf{y},c_2t). 
    \end{split}
\end{equation}
Note that by the fact $d(\mathbf{x},\mathbf{y})\leq c_3\sqrt{t}$, we have $e^{-c_{\rm l}\frac{d(\mathbf{x},\mathbf{y})^2}{t}} \geq C>0$. Further, by the doubling property of $dw$ and~\eqref{eq:Lambda_2t}, if $\varepsilon>0$ is small enough,~\eqref{eq:pert_app} implies
\begin{align*}
     k^{\{V\}}_t(\mathbf{x},\mathbf{y}) \geq cw(B(\mathbf{x},\sqrt{t}))^{-1}\Lambda(\mathbf{x},\mathbf{y},t)
\end{align*}
for some constant $c>0$. Finally,~\eqref{eq:small_t_fin} is a consequence of~\eqref{eq:main_claim}.
\end{proof}

\begin{proposition}\label{prop:lower_small}
Assume that $V:\mathbb{R}^N \longmapsto [0,\infty)$ is continuous, bounded, and Green bounded. There are $\varepsilon,c_4,C_4>0$ such that if $\| \mathcal G(V)\|_{L^\infty}<\varepsilon$, then for all $\mathbf{x},\mathbf{y} \in \mathbb{R}^N$ and $t>0$ one has
\begin{equation}
    k^{\{V\}}_t(\mathbf{x},\mathbf{y})\geq C_4 h_{c_4 t}(\mathbf{x},\mathbf{y}).
\end{equation}
\end{proposition}

\begin{proof}
Let  $c_3=1$, $\tilde c_3$ and $\varepsilon>0$  be  as in Corollary~\ref{eq:lower_small}. Without loss of generality we assume that  $t=1$. Further, according to Corollary~\ref{eq:lower_small}, it suffices to consider $d(\mathbf x,\mathbf y)\geq 1$. Let $\sigma\in G$ be such that 
$ \|\sigma(\mathbf y)-\mathbf x\|=d(\mathbf y,\mathbf x)$. Set $\mathbf{y}'=\sigma(\mathbf{y})$, $n=64[d(\mathbf{x},\mathbf{y})^2]=64[\| \mathbf{x}-\mathbf{y}'\|^2]$, and 
\begin{align*}
    \mathbf{y}_j=\mathbf{x}+j\frac{\mathbf{y}'-\mathbf{x}}{n}, \quad j=1,2,\ldots,n-1.
\end{align*}
Consider the balls $B_j= B\left(\mathbf{y}_j, (8\sqrt{n})^{-1}\right)$.
By the semigroup property of $k^{\{V\}}_{1}(\mathbf{x},\mathbf{y})$ and the fact that $k^{\{V\}}_{t_1}(\mathbf{x}_1,\mathbf{x}_2) \geq 0$ for all $\mathbf{x}_1,\mathbf{x}_2 \in \mathbb{R}^N$  and $t_1>0$ we have
\begin{equation}\label{eq:iter}
    \begin{split}
     &k^{\{V\}}_1(\mathbf{x},\mathbf{y})\\&=\int_{\mathbb{R}^N}\int_{\mathbb{R}^N} \ldots \int_{\mathbb{R}^N}\int_{\mathbb{R}^N} k^{\{V\}}_{\frac{1}{n} }(\mathbf{x},\mathbf{z}_1)   k^{\{V\}}_{\frac{1}{n} }(\mathbf{z}_1,\mathbf{z}_2) \ldots k^{\{V\}}_{\frac{1}{n} }(\mathbf{z}_{n-2},\mathbf{z}_{n-1})  k^{\{V\}}_{\frac{1}{n} }(\mathbf{z}_{n-1},\mathbf{y})\, dw(\mathbf{z}_1)\ldots\, dw(\mathbf{z}_{n-1})\\
     &\geq \int_{B_1}\int_{B_2}\ldots\int_{B_{n-2}}\int_{B_{n-1}}k^{\{V\}}_{\frac{1}{n} }(\mathbf{x},\mathbf{z}_1)   k^{\{V\}}_{\frac{1}{n} }(\mathbf{z}_1,\mathbf{z}_2) \ldots k^{\{V\}}_{\frac{1}{n} }(\mathbf{z}_{n-2},\mathbf{z}_{n-1})  k^{\{V\}}_{\frac{1}{n} }(\mathbf{z}_{n-1},\mathbf{y})\, dw(\mathbf{z}_{n-1})\ldots\, dw(\mathbf{z}_{1}). 
    \end{split}
\end{equation}
Observe that for $\mathbf{z}_j \in B_{j}$ and $\mathbf{z}_{j+1} \in B_{j+1}$ we have
\begin{align*}
    \| \mathbf{z}_j-\mathbf{z}_{j+1}\|\leq \frac{4}{8\sqrt{n}}.    
\end{align*}
By Corollary~\ref{eq:lower_small}, Lemma~\ref{lem:small_deformation},~\eqref{eq:main_lower}, and the doubling property of the measure $dw$, we get
\begin{equation}\label{eq:mult_iter}
    k^{\{V\}}_{\frac{1}{n}}(\mathbf{z}_j,\mathbf{z}_{j+1})\geq c_5 w\left(B\left(\mathbf{y}_j,(\sqrt{n})^{-1}\right)\right)^{-1}.
\end{equation}
Moreover, by the fact that $d(\mathbf{z}_{n-1},\mathbf{y})=d(\mathbf{z}_{n-1},\mathbf{y}') \leq \|\mathbf{z}_{n-1}-\mathbf{y}'\| \leq \frac{4}{8\sqrt{n}}$, Corollary~\ref{eq:lower_small}, and Lemma~\ref{lem:small_deformation} (with $c_0=2$), we obtain 
\begin{equation}\label{eq:single_iter}
    k^{\{V\}}_{\frac{1}{n} } (\mathbf{z}_{n-1}, \mathbf{y}) \geq \widetilde{c}_3 h_{\frac{1}{n}}(\mathbf{z}_{n-1}, \mathbf{y})\geq \widetilde{c}_3 { C_0^{-1}} h_{{ 1/(2n)}}(\mathbf{y},\mathbf{y}').
\end{equation}
Recall that by the doubling property of $dw$ and the definition of $B_j$ we have
\begin{align*}
    \frac{w(B_{j})}{w\left(B\left(\mathbf{y}_j,(\sqrt{n})^{-1}\right)\right)} \geq c.
\end{align*}
Therefore, by~\eqref{eq:iter},~\eqref{eq:mult_iter}, and~\eqref{eq:single_iter}, for a constant $c_6>0$ small enough,
\begin{equation}\label{eq:iter_collected}
    \begin{split}
     k^{\{V\}}_1(\mathbf{x},\mathbf{y}) 
    \geq c_6^{n-1}h_{ 1/(2n)}(\mathbf{y},\mathbf{y}'). 
    \end{split}
\end{equation}
Then, by~\eqref{eq:main_lower}, doubling property of $dw$,~\eqref{eq:Lambda_2t},~\eqref{eq:balls_asymp}, and the fact that  $d(\mathbf{y},\mathbf{y}')=0$, one gets 
\begin{equation}\label{eq:iter_scaled}
    h_{{1/(2n)}}(\mathbf{y},\mathbf{y}') \geq C_{l}w(B(\mathbf{y},{(\sqrt{2n}})^{-1})^{-1}\Lambda(\mathbf{y},\mathbf{y}', (1/(2n))) \geq Cn^{-\mathbf{N}/2}n^{-{2}|G|}w(B(\mathbf{y},1))^{-1}\Lambda(\mathbf{y},\mathbf{y}',1).
\end{equation}
Recall that $n=64[d(\mathbf{x},\mathbf{y})^2]$.  
Hence, by~\eqref{eq:iter_collected},~\eqref{eq:iter_scaled}, and~\eqref{eq:main_claim}, we obtain 
\begin{align*}
    k^{\{V\}}_1(\mathbf{x},\mathbf{y}) \geq C c_6^{n-1}n^{-\mathbf{N}/2}n^{-{2}|G|}w(B(\mathbf{y},1))^{-1}\Lambda(\mathbf{y},\mathbf{y}',1) \geq {c} e^{-c_7d(\mathbf{x},\mathbf{y})^2}h_{1}(\mathbf{y},\mathbf{y}').
\end{align*}
Further, by Proposition~\ref{propo:glue_together}, we get
\begin{align*}
    e^{-c_7d(\mathbf{x},\mathbf{y})^2}h_{1}(\mathbf{y},\mathbf{y}') \geq C e^{-c_7d(\mathbf{x},\mathbf{y})^2} w(B(\mathbf{x},1)) h_{1/c_1}(\mathbf{x},\mathbf{y})h_{1/c_1}(\mathbf{x},\mathbf{y}').
\end{align*}
Since $\|\mathbf{x}-\mathbf{y}'\|=d(\mathbf{x},\mathbf{y})$, by Lemma~\ref{lem:realization_of_d} and the definition of $\Lambda(\cdot,\cdot,\cdot)$ (see~\eqref{eq:Lambda_def}), we have {$\emptyset \in \mathcal{A}(\mathbf{x},\mathbf{y}')$, so $\Lambda (\mathbf x,\mathbf y', 1/c_1) \geq 1$}. Hence, by \eqref{eq:main_lower}, one obtains 
\begin{align*}
    h_{1/c_1}(\mathbf{x},\mathbf{y}') \geq C w(B(\mathbf{x},1))^{-1}e^{-c_{8}d(\mathbf{x},\mathbf{y})^2}. 
\end{align*}
Thus, using Theorem~\ref{teo:1}, we conclude  that 
\begin{align*}
    k^{\{V\}}_1(\mathbf x,\mathbf y)\geq Ce^{-(c_7+c_8)d(\mathbf x,\mathbf y)^2} h_{1/c_1}(\mathbf x,\mathbf y)\geq C w(B(\mathbf x,\sqrt{1/{c_1}}))^{-1} e^{-c_9d(\mathbf x,\mathbf y)^2}\Lambda (\mathbf x,\mathbf y,1/{c_1}).
\end{align*}
Finally the claim follows by applying \eqref{eq:Lambda_2t} together with Theorem~\ref{teo:1}.
\end{proof}

Let us note that implication~\eqref{item:potential_Green_bounded}$\implies$~\eqref{item:h_ct_leq_k_t} is already proved under the assumption that $\|\mathcal{G}(V)\|_{L^{\infty}}$ is small enough and $V$ is continuous and bounded. In Proposition~\ref{propo:V_bounded}, we will make use of the  Feynman--Kac formula to relax the assumption $\|\mathcal{G}(V)\|_{L^{\infty}}<\varepsilon$ for continuous and bounded functions $V$. Finally, in the further part of this subsection, we will relax the assumption that $V$ is continuous and bounded. 

\begin{proposition}\label{propo:V_bounded}
Assume that $V:\mathbb{R}^N \longmapsto [0,\infty)$ is continuous,  bounded, and Green bounded. Then there are constants $C_{\mathcal{G}},c_{\mathcal{G}}>0$, $C_{\mathcal{G}}<1$, which depend only on the bound of  $\|\mathcal G(V)\|_{L^\infty}$ such that for all $\mathbf{x},\mathbf{y} \in \mathbb{R}^N$ and $t>0$ we have
\begin{equation}
    k^{\{V\}}_t(\mathbf x,\mathbf y)\geq C_{\mathcal{G}} h_{c_{\mathcal{G}}t}(\mathbf x, \mathbf y).
\end{equation}
\end{proposition}
\begin{proof} 
Let $\varepsilon$ be the same as in Proposition \ref{prop:lower_small}. We may assume that $\|\mathcal G(V)\|_{L^\infty}\geq \varepsilon$. Let $1<p<\infty$ be such that $ \|\mathcal G(\frac{1}{p}V)\|_{L^\infty}=\varepsilon /2$. Recall that $k_t^{\{V\}}(\mathbf x,\mathbf y)$ is a continuous function (see Theorem \ref{teo:continuous_k}). By the Lebesgue differentiation theorem, for all $(\mathbf x_0,\mathbf y_0) \in \mathbb{R}^N \times \mathbb{R}^N$ and $t>0$, we have 
\begin{equation}
\begin{split}
k^{\{\frac{1}{p}V\}}_t(\mathbf x_0,\mathbf y_0)&=\lim_{r\to 0^+} \frac{1}{w(B(\mathbf y_0, r))} \int_{B(\mathbf y_0,r)} k^{\{\frac{1}{p}V\}}_t(\mathbf x_0,\mathbf y)\, dw(\mathbf y)\\
&=\lim_{r\to 0^+} \frac{1}{w(B(\mathbf y_0, r))} E^{\mathbf x_0}\Big[\exp\Big({-\int_0^t \frac{1}{p} V(X_s)\, ds}\Big)\chi_{B(\mathbf y_0, r) } (X_t)\Big],
\end{split}
\end{equation}
where in the last equality we have used the Feynman-Kac formula~\eqref{eq:Feynman-Kac}. Now, we apply the H\"older's inequality with the exponents $p+p'=pp'$, and then Theorem~\ref{teo:1} obtaining 
\begin{equation}\label{eq:tilde_k1}
\begin{split}
k^{\{\frac{1}{p}V\}}_t(\mathbf x_0,\mathbf y_0)&\leq \lim_{r\to 0^+} \frac{1}{w(B(\mathbf y_0, r))} \Big\{E^{\mathbf x_0}\Big(e^{-\int_0^t  V(X_s)\, ds}\chi_{B(\mathbf y_0, r) } (X_t)\Big)\Big\}^{1/p}  \Big\{E^{\mathbf x_0}\Big(\chi_{B(\mathbf y_0, r) } (X_t)\Big)\Big\}^{1/p'}\\
& =\Big\{k^{\{V\}}_t(\mathbf x_0,\mathbf y_0)\Big\}^{1/p} h_t(\mathbf x_0,\mathbf y_0)^{1/p'}\\
&\leq C \Big\{k^{\{V\}}_t(\mathbf x_0,\mathbf y_0)\Big\}^{1/p} \frac{\Lambda(\mathbf x_0,\mathbf y_0,t)^{1/p'}}{w(B(\mathbf x_0,\sqrt{t}))^{1/p'}}. 
\end{split}
\end{equation} By Proposition~\ref{prop:lower_small},~\eqref{eq:main_lower}, the doubling property \eqref{eq:doubling}, and~\eqref{eq:Lambda_2t},  we have 
\begin{equation}\label{eq:tilde_k2}
    \begin{split}
        k^{\{\frac{1}{p}V\}}_t(\mathbf x_0,\mathbf y_0)\geq C_4 h_{c_4 t}(\mathbf x_0,\mathbf y_0) & \geq c''w(B(\mathbf x_0,\sqrt{t}))^{-1} e^{-c'd(\mathbf x_0,\mathbf y_0)^2/4t}\Lambda (\mathbf x_0,\mathbf y_0,t).
    \end{split}
\end{equation}
Thus, combining \eqref{eq:tilde_k1} together with \eqref{eq:tilde_k2}, we get 
\begin{equation}
    \begin{split}
    c''w(B(\mathbf x_0,\sqrt{t}))^{-1} e^{-c'd(\mathbf x_0,\mathbf y_0)^2/4t} \Lambda(\mathbf x_0,\mathbf y_0,t)\leq    C \Big\{k^{\{V\}}_t(\mathbf x_0, \mathbf y_0)\Big\}^{1/p} \frac{\Lambda(\mathbf x_0,\mathbf y_0,t)^{1/p'}}{w(B(\mathbf x_0,\sqrt{t}))^{1/p'}}
    \end{split}.
\end{equation}
Finally, by Theorem~\ref{teo:1},
\begin{equation}
    k^{\{V\}}_t(\mathbf x_0,\mathbf y_0)\geq \Big(\frac{c''}{C}\Big)^p e^{-c'pd(\mathbf x_0,\mathbf y_0)^2/4t}  \frac{\Lambda(\mathbf x_0,\mathbf y_0,t)}{w(B(\mathbf x_0,\sqrt{t}))}\geq ch_{t/c'} (\mathbf x_0,\mathbf y_0).
\end{equation}
\end{proof} 

\begin{proposition}\label{propo:V_bounded_2}
Assume that $V \colon \mathbb{R}^N \longmapsto [0,\infty)$ is measurable, bounded, and Green bounded. Then there are constant $\widetilde{C}_{\mathcal{G}},\widetilde{c}_{\mathcal{G}}>0$, which depend only on the bound of $\|\mathcal G(V)\|_{L^\infty }$, such that for all $\mathbf{x},\mathbf{y} \in \mathbb{R}^N$ and $t>0$ we have
\begin{equation}\label{eq:lower_V_bounded}
    k^{\{V\}}_t(\mathbf x,\mathbf y)\geq \widetilde{C}_{\mathcal{G}} h_{\widetilde{c}_{\mathcal{G}}t}(\mathbf x, \mathbf y).
\end{equation}
\end{proposition}
\begin{proof}
For $n \in \mathbb{N}$ we consider 
\begin{align*}
    \widetilde{V}_n(\mathbf x)=\int_{\mathbb{R}^N} h_{1/n}(\mathbf x,\mathbf y)V(\mathbf y)\, dw(\mathbf y).    
\end{align*}
Then,  $ \lim_{n \to \infty}  \widetilde{V}_n(\mathbf x)=V(\mathbf x)$  for almost all $\mathbf x \in \mathbb{R}^N$ (see e.g.~\cite[Remark 5.5]{ADzH}) and, by the regularity of the heat semigroup, $\widetilde{V}_n$ are continuous functions. Moreover, by~\eqref{eq:probabilistic}, $\| \widetilde{V}_n\|_{L^{\infty}} \leq \| V\|_{L^{\infty}}$, and, by Proposition~\ref{prop:Greens}, there is a constant $C>0$ such that
\begin{align*}
    \|\mathcal G (\widetilde{V}_n)\|_{L^\infty}\leq C \| \mathcal G(V) \|_{L^\infty} .   
\end{align*} 
Recall that $\{e^{-t\tilde L_n}\}_{t \geq 0}$ and $\{e^{-tL}\}_{t \geq 0}$ are the contraction  semigroups on $L^2(dw)$ generated by the operators $-\tilde L_n=\Delta_k-\widetilde{V}_n$ and $-L=\Delta_k-V$ respectively. Then, for $f\in \mathcal D(L)=\mathcal D(\tilde L_n)=\mathcal D(\Delta_k)$,  we have $\lim_{n\to\infty} \tilde L_nf=Lf$. Hence Theorem 3.4.5 of \cite{Pazy} asserts that 
\begin{equation}\label{eq:L2-convergence}
\lim_{n\to \infty}  e^{-t\tilde L_n}f= e^{-tL}f \quad \text{in } {L^2(dw)}-{\rm norm}, \quad {\text{for all }} f\in L^2(dw).
\end{equation}
Further,  Proposition~\ref{propo:V_bounded} and Theorem~\ref{teo:continuous_k} imply that there are constants $C_{\mathcal{G}},c_{\mathcal{G}}>0$, $C_{\mathcal{G}}<1$, such that for all $n$,  we have 
\begin{align}\label{eq:lower_V_n}
    C_{\mathcal{G}}h_{c_{\mathcal{G}} t}(\mathbf x,\mathbf y)\leq k^{\{\widetilde{V}_n\}}_t(\mathbf x,\mathbf y) \text{ for all }(t,\mathbf x,\mathbf y)\in (0,\infty)\times  \mathbb R^N\times \mathbb R^N. 
\end{align}
Assume towards contradiction that $C_{\mathcal{G}}h_{c_{\mathcal{G}} t}(\mathbf x_0,\mathbf y_0)> k^{\{V\}}_t(\mathbf x_0,\mathbf y_0)$ for some $(t_0, \mathbf{x}_0,\mathbf{y}_0) \in (0,\infty)\times \mathbb{R}^N \times \mathbb{R}^N$. Then, by the fact that $k^{\{V\}}_t(\cdot,\cdot)$ and $h_t(\cdot,\cdot)$ are continuous, there are $\varepsilon, \delta>0$ such that $C_{\mathcal{G}}h_{c_{\mathcal{G}} t_0}(\mathbf x,\mathbf y)> k^{\{V\}}_{t_0}(\mathbf x,\mathbf y)+\varepsilon$ for all $(\mathbf{x},\mathbf{y}) \in B(\mathbf{x}_0,\delta) \times B(\mathbf{y}_0,\delta)$. Hence, applying~\eqref{eq:L2-convergence} to $f=\chi_{B(\mathbf{y}_0,\delta)}$ we obtain a contradiction.
\end{proof}

\begin{proof}[Proof of the  implication\texorpdfstring{~\eqref{item:potential_Green_bounded}$\implies$~\eqref{item:h_ct_leq_k_t}}{(c)=>(a)}]
Assume that $V\colon\mathbb{R}^N \longmapsto [0,\infty)$, $V \in L^1_{{\rm loc}}(dw)$, is Green bounded. Consider the operators $L_n=-\Delta_k+V_n$, $V_n=\min (V, n)$, $n \in \mathbb{N}$.  By Proposition~\ref{propo:V_bounded_2}  there are $\widetilde{C}_{\mathcal{G}},\widetilde{c}_{\mathcal{G}}>0$ such that for all $n \in \mathbb{N}$ we have
 \begin{align*}
      k_t^{\{V_n\}}(\mathbf x ,\mathbf y) \geq \widetilde{C}_{\mathcal{G}} h_{\widetilde{c}_{\mathcal{G}}t}(\mathbf x,\mathbf y)
 \end{align*}
for all $(\mathbf{x},\mathbf{y}) \in \mathbb{R}^N \times \mathbb{R}^N$. Now the required lower  bound for $k^{\{V\}}_t(\mathbf x,\mathbf y)$ follows from Theorem~\ref{teo:kernel_etL}.

\end{proof}

\appendix

\section{Proof of Proposition \ref{prop:integral_cadlag}}\label{sec:app}

\begin{lemma}\label{lem:J_f}
Assume that $f:[a,b]\to \mathbb R$ is a bounded c\`adl\`ag function. Define \begin{equation}
    J_f(t_0):=\big|\lim_{t\to t_0^{-}}f(t)-\lim_{t\to t_0^+} f(t)\big|= \big|\lim_{t\to t_0^{-}}f(t)-f(t_0)\big|.    
\end{equation}
Then, for all 
$ \varepsilon >0$, one has 
$$ \#\{t\in [a,b]: J_f(t)\geq \varepsilon \}= C_\varepsilon <\infty. $$  
\end{lemma}
\begin{proof}
Aiming for a contradiction, suppose that the set $ A=\{t\in [0,1]: J_f(t)\geq \varepsilon \}$ is infinite.  Let $t_0$ be a accumulation point of $A$.  There is $\delta>0$ such that $|f(t)-\lim_{t\to t_0^{-}}f(t)|<\varepsilon /4$ for 
$t_0-\delta<t<t_0.$   Thus $|f(t)-f(t')|\leq \varepsilon /2$ for $t_0-\delta <t,t'<t_0$. We proceed similarly to obtain 
$|f(t)-f(t')|\leq \varepsilon /2$ for $t_0<t,t'<t_0+\delta'$. 
So $J_f(t)\leq \varepsilon /2$ for $t\in (t_0-\delta, t_0+\delta')$, and we get the contradiction. 
\end{proof}

\begin{proof}[Proof of Proposition \ref{prop:integral_cadlag}] We may assume that $a=0$, $b=1$, and $|f(t)|\leq 1$. 
Fix $\varepsilon >0$. Consider the  finite set
\begin{align*}
   A= \{t\in [0,1]: J_f(t)\geq \varepsilon \}=\{t_1,t_2,\ldots,t_{m-1}\}    
\end{align*}
(see Lemma~\ref{lem:J_f}).  Let $U$ be an open set such that $\{t_1,t_2,\ldots,t_{m-1}\} \subset U$, $|U|<\varepsilon$, and $[0,1]\setminus U$ is a finite union of closed disjoint intervals $I_1$,\ldots $I_m$. Then 
$$ \int_U | f(t)|\, dt < \varepsilon.$$
Consider $I_j=[a_j, b_j]$. For every $t\in [a_j,b_j]$ there is $\delta_t>0$ such that $|f(t)-f(t')|<4\varepsilon $ for $t'\in [a_j,b_j]$, $|t'-t|<\delta_t$, because $J_ff(t)<\varepsilon$. By compactness, there is $\delta_j>0$ such that $|f(t)-f(t')|\leq 8\varepsilon $ for all $t,t'\in [a_j,b_j]$, $|t-t'|<\delta_j$. Take $\delta=\min\{\delta_1,\ldots,\delta_m\}$. 
If $n \in \mathbb{N}$ is such that $\frac{1}{n}\leq \delta/2$ and $[\frac{k}{n},\frac{k+1}{n}]\subseteq [a_j,b_j]$, then 
\begin{align*}
    \Big|\int_{k/n}^{(k+1)/n} f(t)\, dt -\frac{1}{n} f(k/n)\Big|<8\varepsilon /n.
\end{align*}
So, we easily conclude that
\begin{align*}
     \Big| \frac{1}{n}\sum_{k=0}^{n-1} f(k/n)- \int_0^1 f(t)\, dt\Big|\leq 20\varepsilon
\end{align*}
for $n \in \mathbb{N}$ large enough. 
\end{proof}

\end{document}